\documentclass{amsart}
\usepackage{geometry}
\usepackage{hyperref}
\usepackage{amssymb}
\usepackage{amsmath}
\usepackage{esint}
\usepackage{leftidx,enumerate}
\usepackage{mathtools,xspace}
\usepackage{MACROS}

\usepackage[usenames,dvipsnames]{color}
\newcommand{\AJS}[1]{{\color{black}#1}}
\newcommand{\rhn}[1]{{\color{black}#1}}
\newcommand{\jp}[1]{{\color{black}#1}}

\newcommand{\eremk}{\hfill\rule{1.2ex}{1.2ex}}

\begin{document}

\title[Fractional obstacle problem]{Weighted Sobolev regularity and rate of approximation of the obstacle problem for the integral fractional Laplacian}

\author[J.P.~Borthagaray]{Juan Pablo~Borthagaray}
\address[J.P.~Borthagaray]{Department of Mathematics, University of Maryland, College Park, MD 20742, USA}
\email{jpb@math.umd.edu}
\thanks{JPB has been supported in part by NSF grant DMS-1411808}

\author[R.H.~Nochetto]{Ricardo H.~Nochetto}
\address[R.H.~Nochetto]{Department of Mathematics and Institute for Physical Science and Technology, University of Maryland, College Park, MD 20742, USA}
\email{rhn@math.umd.edu}
\thanks{RHN has been supported in part by NSF grant DMS-1411808}

\author[A.J.~Salgado]{Abner J.~Salgado}
\address[A.J.~Salgado]{Department of Mathematics, University of Tennessee, Knoxville, TN 37996, USA}
\email{asalgad1@utk.edu}
\thanks{AJS is supported by NSF grant DMS-1720213}




\begin{abstract}
We obtain regularity results in weighted Sobolev spaces for the solution of the obstacle problem for the integral fractional Laplacian $(-\Delta)^s$ \rhn{in a Lipschitz bounded domain $\Omega\subset\R^n$ satisfying the exterior ball condition.} The weight is a power of the distance to the boundary \rhn{$\partial\Omega$ of $\Omega$} that accounts for the singular boundary behavior of the solution for any $0<s<1$. These bounds then serve us as a guide in the design and analysis of \rhn{a finite element scheme over graded meshes for any dimension $n$, which is optimal for $n=2$.}
\end{abstract}

\maketitle

\section{Introduction}
\label{sec:Intro}

The purpose of this work is, ultimately, the design of an optimally convergent finite element method for the solution of the obstacle problem for the integral fractional Laplacian which, from now on, we shall simply refer to as the fractional obstacle problem. In addition to the intrinsic interest that the study of unilateral problems with nonlocal operators may give rise to, the fractional obstacle problem appears in the study of systems of particles with strong (non Newtonian) repulsion \cite{CarrilloDelgadinoMellet,Serfaty} and of
optimal stopping times for jump processes (see \cite{MR2256030} and \cite[Chapter 10]{MR2001996}). The latter, in particular, is used in the modeling of the rational price of a perpetual American option \cite{ContTankov}. We also refer the reader to \cite{RosOtonSerra-nonlinear, Silvestre} for an account of other applications.

To make matters precise, here we describe the (eventually equivalent) formulations that the fractional obstacle problem may be written as. For $n \geq 1$ we let $\W \subset \rn$ \AJS{be a bounded domain} with Lipschitz boundary $\pp\W$ that satisfies the exterior ball condition. For two functions $f: \W \to \R$ and $\chi : \overline\W \to \R$, with $\chi < 0$ on $\pp\W$, and $s \in (0,1)$ we seek a function $u: \rn \to \R$ such that $u = 0$ in $\cW = \rn\setminus\Omega$ and it satisfies the complementarity system
\begin{equation}
\label{eq:complementarity}
  \min\left\{ \lambda, u-\chi \right\} = 0, \mae \Omega, \qquad \lambda := \Laps u - f.
\end{equation}
This problem can also be written as a constrained minimization problem on the space $\tHs$ (see section~\ref{sec:notation} for notation). Indeed, if we define the set of admissible functions 
\begin{equation}
\label{eq:def_K}
  \K = \left\{ v \in \tHs: v \geq \chi \mae \Omega \right\},
\end{equation}
then the solution to the fractional obstacle problem can also be characterized as the (unique) minimizer of the functional
\[
  \mathcal{J}: v \mapsto \mathcal{J}(v) = \frac12 | v |_{\tHs}^2 - \langle f, v \rangle,
\]
over the convex set $\K$. Equivalently, this minimizer $u \in \K$ solves the variational inequality
\begin{equation}
\label{eq:obstacle}
  (u,u-v)_s \leq \langle f, u-v \rangle, \quad \forall v \in \K,
\end{equation}
\AJS{where by $(\cdot,\cdot)_s$ we denote the inner product on $\tHs$ induced by the fractional Laplacian \rhn{(see \eqref{eq:defofinnerprod}), and $\langle\cdot,\cdot\rangle$ is the duality pairing between $\tHs$ and its dual $H^{-s}(\Omega)$}}.
We refer the reader to section~\ref{sub:obstacle_known} and \cite{MR3630405} for a more thorough exploration of these formulations and their equivalence. Finally we must mention that although in bounded domains there are many, nonequivalent, definitions of the operator $\Laps$, motivated by applications, here we choose the so-called integral one; that is, for a sufficiently smooth function $v:\rn \to \R$ weh set
\begin{equation}
\label{eq:defofLaps}
  \Laps v(x) = C(n,s) \mbox{ p.v.} \int_\rn \frac{v(x)-v(y)}{|x-y|^{n+2s}} \diff y, \qquad C(n,s) = \frac{2^{2s} s \Gamma(s+\frac{n}{2})}{\pi^{n/2}\Gamma(1-s)}.
\end{equation}

Our choice of definition is justified by the fact that, unlike the regional or the spectral ones, the integral fractional Laplacian of order $s$ is the infinitesimal generator of a $2s$-stable L\'evy process. \AJS{In this context, working on a bounded domain would correspond to a so-called \emph{killed process}, that is one that finishes upon exiting the domain.
L\'evy} processes have been widely employed for modeling market fluctuations, both for risk management and option pricing purposes. It is in this context that, as mentioned above, the fractional obstacle problem arises as a pricing model for American options. More precisely, if $u$ represents the rational price of a perpetual American option, modeling the assets prices by a L\'evy process $X_t$ and denoting by $\chi$ the payoff function,  then $u$ solves \eqref{eq:obstacle}. We refer the reader to \cite{ContTankov} for an overview of the use of jump processes in financial modeling.

Taking into account their applications in finance, it is not surprising that numerical schemes for integro-differential inequalities have been proposed and analyzed in the literature; we refer the reader to \cite{HRSW:09} for a survey on these methods. These applications aim to approximate the price of a number of assets; therefore, the consideration of a logarithmic price leads to problems posed in the whole space $\rn$. 
For the numerical solution, it is usual to perform computations on a sufficiently large tensor-product domain.
Among the schemes based on Galerkin discretizations, reference \cite{WW:08} utilizes piecewise linear Lagrangian finite elements, while \cite{MNS:05} proposes the use of wavelet bases in space.
As for approximations of variational inequalities involving integral operators on arbitrary bounded domains, an a posteriori error analysis is performed in \cite{NvPZ:10}.

Since the seminal work of Silvestre \cite{Silvestre}, the fractional obstacle problem started to draw the attention of the mathematical community. 
Using potential theoretic methods, reference \cite{Silvestre} shows that if the obstacle is of class $C^{1,s}$, then the solution to the fractional obstacle problem is of class $C^{1,\alpha}$ for all $\alpha \in (0,s)$; optimal $C^{1,s}$ regularity of solutions was derived assuming convexity of the contact set.
The pursuit of the optimal regularity of solutions without a convexity hypothesis, in turn, motivated the celebrated extension by Caffarelli and Silvestre \cite{CaffarelliSilvestre} for the fractional Laplacian in $\rn$. Using this extension technique, Caffarelli, Salsa and Silvestre proved, in \cite{CaffarelliSalsaSilvestre}, the optimal regularity of solutions (\cf Proposition \ref{prop:localHolder} below). It is important to notice, however, that this is only an interior regularity result. Nothing is said about the boundary behavior of the solution to \eqref{eq:obstacle}. This is a highly nontrivial issue, as it is known that even the solution to a linear problem involving the fractional Laplacian on a very smooth domain possesses limited regularity near the boundary; see \cite{Grubbboundary,Grubb} and section~\ref{sub:linear} below for details. In addition, regularity results in H\"older spaces are not amenable to the development of an error analysis for a finite element method.

Using the extension technique, one could in principle follow the lines of \cite[Section 2]{CaffarelliSalsaSilvestre} to obtain, via a localization argument, regularity results for the obstacle problem posed on a bounded domain. This would entail dealing with a degenerate elliptic equation where the weight belongs to the Muckenhoput class $A_2$. We could then invoke the results from \cite{FabesKenigSerapioni,MENGESHA2019184} and the translation invariance in the $x$-variable of the extension weight to conclude the desired regularity. While accomplishing this program seems possible, it would only yield results for the fractional Laplacian, and the techniques would not extend to more general nonlocal operators, like those studied in \cite{CaffarelliRosOtonSerra}.

Our regularity approach is entirely nonlocal and based on localization without invoking the extension. However, we must immediately point out that if $0\le\eta\le1$ is a smooth cut-off function, then
\[
  \Laps (\eta u) \ne \eta \Laps u \quad \mbox{in } \{ \eta = 1 \}
\]
because of the nonlocal structure of $\Laps$. Consequently, we cannot deduce regularity of $\eta u$ directly from that of $\Laps u$. This is one of the main technical difficulties we overcome in this work.

In this paper, \AJS{under certain smoothness and compatibility assumptions on the forcing $f$ and the obstacle $\chi$ (see \eqref{eq:defofcalF} for a precise statement)}, we combine H\"older estimates from \cite{CaffarelliSalsaSilvestre,Silvestre} and \cite{RosOtonSerra} to derive interior and boundary H\"older estimates for \eqref{eq:obstacle}. This is achieved under a nondegeneracy condition: the obstacle needs to be negative near the boundary. In this case, the solution to \eqref{eq:obstacle} behaves, essentially, like the solution to a linear problem near the boundary, for which the H\"older regularity is known \cite{RosOtonSerra}. We then follow ideas from \cite{AcostaBorthagaray} to derive global regularity results in weighted Sobolev spaces, which guide us in the design of an optimally convergent finite element scheme over graded meshes. These meshes compensate for the singular boundary behavior of the solution of \eqref{eq:obstacle} regardless of the fractional order $s\in(0,1)$. We discuss their design and derive a quasi-optimal rate of convergence in the natural energy norm.

We must comment that a related numerical analysis for the obstacle problem, corresponding to the \emph{spectral} fractional Laplacian, was carried out in \cite{MR3393323}; we refer the reader to \cite{BBNOS} for a comparison between these operators and a survey of numerical methods for fractional diffusion. The recent work \cite{BG18} also deals with finite element approximations to nonlocal obstacle problems, involving both finite and infinite-horizon kernels. Experiments, carried out for one-dimensional problems with uniform meshes, indicate convergence with order $h^{1/2}$ in the energy norm. However, \cite{BG18} does not provide an error analysis for the nonlocal obstacle problem. In this paper we show that using suitably graded meshes essentially \emph{doubles} the convergence rate in the energy norm. Moreover, a standard argument allows us to extend the results we obtain in this work to nonlocal operators with finite horizon. Finally, we comment that \cite{MR3090147} provides regularity results of Lewy--Stampacchia type for the fractional Laplacian. Their use in a numerical setting, however, is not immediate.
 
The paper is organized as follows. In section~\ref{sec:notation} we set notation and assumptions employed in the rest of the work, and review preliminary results about solutions of the linear Dirichlet problem for the fractional Laplacian on bounded domains and the fractional obstacle problem. These results are employed in section~\ref{sec:regularity} to derive weighted Sobolev regularity estimates for solutions of problem \eqref{eq:obstacle}. Then, section~\ref{sec:FEM} applies our regularity estimates to deduce a quasi-optimal convergence rate for a finite element approximation of the fractional obstacle problem \eqref{eq:obstacle} over graded partitions of bounded polytopal domains. This requires the study of a positivity preserving quasi-interpolation operator in weighted fractional Sobolev spaces; this novel development is carried out in section~\ref{sub:positive}. Finally, numerical examples presented in section~\ref{sec:numex} illustrate the sharpness of our theoretical results and reveal some qualitative properties of the coincidence set.

\section{Notation and preliminaries}
\label{sec:notation}

In this section we will introduce some notation and the set of assumptions that we shall operate under. For $n \geq 1$ we let $\W \subset\rn$ be \AJS{a bounded domain} with Lipschitz boundary $\pp\W$ that satisfies the exterior ball condition. The complement of $\W$ will be denoted by $\cW$ and the fractional order by $s \in (0,1)$. The ball of radius $R$ and center $x \in \rn$ will be denoted by $B_R(x)$, and we set $B_R = B_R(0)$. During the course of certain estimates we shall denote by $\omega_{n-1}$ the $(n-1)$-dimensional Hausdorff measure of the unit sphere $\partial B_1$. As usual, we will denote by $C$ a nonessential constant, and its specific value might change from line to line. By $C(A)$ we shall mean a nonessential constant that may depend on $A$. Finally, by $A \approx B$ we mean that $A \leq C B$ and $B \leq C A$.

Unless indicated otherwise, we will follow standard notation regarding function spaces. In particular, \AJS{for a bounded domain $D \subset \rn$, $k \in \N \cup \{0\}$, and $\gamma \in [0,1]$, we denote
\[
  C^{k,\gamma}(\overline D) = \left\{ w \in C^k( \overline D) : 
  \sup_{x,y \in \overline D, x \neq y} \max_{\beta \in (\N \cup\{0\})^n: |\beta| = k} \frac{|\partial^\beta w(x)- \partial^\beta w(y)|}{|x-y|^\gamma} < \infty \right \}.
\]
In addition $w \in C^{k,\gamma}(D)$ if $w \in C^{k,\gamma}(\overline U)$ for all $U \Subset D$.}
\rhn{The Sobolev space of order $s>0$ over $\rn$ is defined as}
\[
  H^s(\rn) = \left\{ v \in L^2(\rn): \xi \mapsto (1+|\xi|^2)^{s/2}\mathcal{F}(v)(\xi) \in L^2(\rn) \right\},
\]
with norm
\[
  \| v \|_{H^s(\rn)} = \left\| \xi \mapsto (1+|\xi|^2)^{s/2}\mathcal{F}(v)(\xi) \right\|_{L^2(\rn)}.
\]
In these definitions $\mathcal{F}$ denotes the Fourier transform. The closure of $C_0^\infty(\Omega)$ in $H^s(\rn)$ will be denoted by $\tHs$. This space can also be characterized as follows:
\begin{equation}
\label{eq:defoftHs}
  \tHs := \left\{ v_{|\Omega} : v \in H^s(\rn), \ \supp v \subset \overline \W \right\}.
\end{equation}
We comment that, on $\tHs$, the natural inner product is equivalent to
\begin{equation}
\label{eq:defofinnerprod}
  (v,\phii)_s = \frac{C(n,s)}2 \iint_{\rn \times \rn} \frac{(v(x)-v(y))(\phii(x) - \phii(y))}{|x-y|^{n+2s}} \diff x \diff y, \qquad |v|_{\tHs} = (v,v)_s^{1/2}.
\end{equation}
The duality pairing between $\tHs$ and its dual $H^{-s}(\Omega)$ is denoted by $\langle \cdot , \cdot \rangle$. In view of \eqref{eq:defofinnerprod} we see that, whenever $v \in \tHs$ then $\Laps v \in H^{-s}(\Omega)$ and that
\begin{equation}
\label{eq:innerprodisduality}
  (v,\phii)_s = \langle \Laps v, \phii \rangle, \quad \forall \phii \in C_0^\infty(\W).
\end{equation}

In section~\ref{sec:regularity} it will become necessary to characterize the behavior of the solution to \eqref{eq:obstacle} near the boundary. To do so, we must introduce weighted Sobolev spaces, where the weight is a power of the distance to the boundary. We define
\[
  \delta(x) = \dist(x,\pp\W), \qquad \delta(x,y) = \min\{ \delta(x), \delta(y) \}.
\]
Then, for $k \in \N\cup\{0\}$ and $\alpha \in \R$, we consider the norm
\begin{equation}
\label{eq:defofWnorm}
  \| v \|_{H^k_\alpha(\W)}^2 = \AJS{\sum_{0 \leq |\beta| \leq k} \int_\W|\partial^\beta v(x)|^2 \delta(x)^{2\alpha} \diff x}.
\end{equation}
and define $H^k_\alpha(\W)$ and $\tH^k_\alpha(\W)$ as the closures of $C^\infty(\W)$ and $C_0^\infty(\W)$, respectively, with respect to the norm \eqref{eq:defofWnorm}. We also need to define weighted Sobolev spaces of a non-integer differentiation order, and their zero-trace versions.

\begin{definition}[weighted fractional Sobolev spaces]
\label{def:wfSobolev}
Let $0<t \in \R \setminus \mathbb{Z}$ and $\alpha \in \R$. Assume that $k \in \N \cup\{0\}$ and $\sigma \in (0,1)$ are the unique numbers such that $t = k + \sigma$. The weighted fractional Sobolev space is
\begin{equation*}
  H^t_\alpha (\W) = \left\{ v \in H^k_\alpha(\W) \colon  | \pp^\beta v |_{H^{\sigma}_\alpha (\W)} < \infty \ \forall \beta \in \N^n, \ |\beta| = k \right\} ,
\end{equation*}
where 
\[
| v |^2_{H^{\sigma}_\alpha (\W)} = \iint_{\W\times\W} \frac{|v(x)-v(y)|^2}{|x-y|^{n+2\sigma}} \, \delta(x,y)^{2\alpha} \diff x \diff y.
\]
We endow this space with the norm
\[
  \| v \|_{H^t_\alpha (\W)}^2 = \| v \|_{H_\alpha^k (\W)}^2 + \sum_{|\beta| = k } | \pp^\beta v |^2_{H^{\sigma}_\alpha (\W)} .
\]
Similarly, the zero-trace weighted Sobolev space is
\begin{equation*}
  \tH^t_\alpha (\W) = \left\{ v \in \tH^k_\alpha(\W) \colon | \pp^\beta v |_{H^{\sigma}_\alpha (\rn)} < \infty \ \forall \beta \in \N^n, \ |\beta| = k \right\} ,
\end{equation*}
with the norm
\[
  \| v \|_{\tH^t_\alpha (\W)}^2 = \| v \|_{H_\alpha^k (\W)}^2 + \sum_{|\beta| = k } | \pp^\beta v |^2_{H^{\sigma}_\alpha (\rn)} .
\]
\end{definition}

Spaces like the ones defined above have been considered, for example, in \cite{AcostaBorthagaray} in connection with the study of the regularity properties of the solution to the linear fractional Poisson problem. However, unlike \cite{AcostaBorthagaray}, the spaces $H^t_\alpha(\W)$ and $\widetilde H^t_\alpha(\W)$ require functions to belong respectively to $H^k_\alpha(\W)$ and $\widetilde H^t_\alpha(\W)$, instead of $H^k(\W)$. This is a weaker condition and that shall become important below.

We remark also that, during our discussion, we will make use of the norms and seminorms of $H^t_\alpha(\omega)$ and $\tH^t_\alpha(\omega)$, where $\omega$ is a Lipschitz subdomain of $\W$. If that is the case, the weight $\delta$ will always refer to the distance to $\pp\W$.

As a final preparatory step, we recall an interior regularity result for $s$--harmonic functions over balls.

\begin{lemma}[balayage]
\label{lemma:balayage}
Let $w \in L^\infty(\rn)$ be such that $\Laps w = 0$ in $B_R$. Then, $w \in C^\infty(B_{R/2})$.
\end{lemma}
\begin{proof}
According to \cite[formula (1.6.11')]{Landkof}, in the ball $B_R$, any $s$-harmonic function $w$ can be represented using a Poisson kernel:
\[
  w(x) = \int_{\compl{B_R}} w(y) P(x,y) \diff y,
\]
where
\[
  P(x,y) = C \left( \frac{R^2 - |x|^2}{|y|^2 - R^2} \right)^s \frac{1}{|x-y|^n} .
\]
Consequently, whenever $x \in B_{R/2}$, it is legitimate to differentiate to any order the representation above.
\end{proof}

\subsection{The linear problem}
\label{sub:linear}

Here we consider the linear version of \eqref{eq:obstacle}; that is, we formally set $\chi = - \infty$ to arrive at the problem: given $g \in H^{-s}(\Omega)$ we seek for $w_g \in \tHs$ such that
\begin{equation}
\label{eq:Dirichlet}    
  \Laps w_g = g \text{ in } \W, \qquad w_g = 0 \text{ in } \cW.
\end{equation}
Identity \eqref{eq:innerprodisduality} yields the existence and uniqueness of a solution to this problem. In addition, since the kernel is positive, we have a nonlocal maximum principle.

\begin{proposition}[nonlocal maximum principle]
\label{prop:fractmaxprinciple}
Let $g \in H^{-s}(\W)$ be such that $g \geq 0$ in $\W$, then we have that $w_g \geq 0$ in $\W$.
\end{proposition}
\begin{proof}
See \cite[Proposition 4.1]{MR3447732}.
\end{proof}

The investigation of the regularity of the solution to \eqref{eq:Dirichlet} has been an active area of research in recent years. Solutions to this problem are known to possess limited boundary regularity. Namely, the behavior
\[
w_g(x) \approx \dist(x,\pp\W)^s,
\]
is expected independently of the smoothness of the domain $\W$ and right hand side $g$. Assuming $\W$ is smooth, this behavior can be precisely quantified in terms of H\"ormander regularity \cite{Grubb}; for Lipschitz domains satisfying the exterior ball condition it can also be expressed in terms of the reduced H\"older regularity of solutions \cite{RosOtonSerra},
\[
  \| w_g \|_{C^{\AJS{0,s}}(\rn)} \leq C \| g\|_{L^\infty(\W)}.
\]

If the right hand side $g$ happens to be more regular, then finer estimates on the solution $w_g$ can be derived.

\begin{proposition}[H\"older estimates for the linear problem] \label{prop:Holder}
Let $\Omega$ be a bounded Lipschitz domain satisfying the exterior ball condition. Let $g \in C^{\AJS{0,1-s}}(\overline\Omega)$ and $w_g$ be the solution of \eqref{eq:Dirichlet}. Then, $w_g$ satisfies 
\begin{equation} \label{eq:holder_linear}
\|w_g\|_{C^{\AJS{0,s}}(\overline\Omega)}  + \sup_{x \in \Omega} \delta(x)^{1-s} |\nabla w_g(x)| + \sup_{x,y \in \Omega} \delta(x,y)\frac{|\nabla w_g(x) - \nabla w_g(y)|}{|x-y|^s} \leq C(\Omega,s) \| g \|_{C^{\AJS{0,1-s}}(\overline\Omega)}.
\end{equation}
\end{proposition}
\begin{proof}
  It suffices to set $\beta = 1-s$ in \cite[Proposition 1.4]{RosOtonSerra}.
\end{proof}

For one-dimensional or radial domains, these regularity estimates can be further sharpened by deriving explicit expressions for the map $w \mapsto (-\Delta)^s \left[\dist(\cdot,\pp\W)^s w\right]$ in terms of expansions in bases consisting of special functions, see \cite{ABBM, Dyda2016}.
Of importance in the design of optimally convergent finite element schemes is \cite{AcostaBorthagaray}, where regularity in spaces similar to those introduced in Definition~\ref{def:wfSobolev} (weighted fractional Sobolev spaces) was derived. Below we extend and modify these results to fit the framework that we are adopting here.

\begin{theorem}[weighted regularity of $w_g$]
\label{teo:weighted}
Let $\W$ be a bounded Lipschitz domain satisfying the exterior ball condition. Let $g \in C^{\AJS{0,1-s}}(\overline \W)$ and $w_g$ be the unique solution of \eqref{eq:Dirichlet}. Then, for every \AJS{$\eps \in (0,s/2)$}, we have that $w_g \in \tH^{1+s-2\eps}_{1/2- \eps}(\W)$, with the estimate
\[
  \| w_g \|_{\tH^{1+s-2\eps}_{1/2- \eps}(\W)} \leq \frac{C(\W,s)}{\eps} \| g \|_{C^{\AJS{0,1-s}}(\overline \W)} .
\]
\end{theorem}
\begin{proof}
We must first notice that, as mentioned before, the spaces of Definition~\ref{def:wfSobolev} (weighted fractional Sobolev spaces) do not require integrability of the derivatives of functions with respect to Lebesgue measure but with respect to $\delta^{2\alpha}(x) \diff x$. Since, in this case, $\alpha = 1/2-\eps >0$, this is a weaker condition, as it allows certain blow up of the derivatives near the boundary. Hence, for $s \in (1/2,1)$, the assertion follows from the stronger estimate \cite[\AJS{Proposition 3.16}]{AcostaBorthagaray}. \AJS{A direct estimate can also be obtained \rhn{with} the same arguments used to bound the term $\calI_{\mathcal{O}}$ defined below.}

In the case $s \in (0,1/2]$, \AJS{we begin by observing that $w_g \in C(\overline \W) \subset L^2_{1/2-\eps}(\W)$. In addition,} the middle term in estimate \eqref{eq:holder_linear} implies that
\[
  \| w_g \|_{H^1_{1/2-\eps}(\Omega)}^2 = \int_\Omega |\nabla w_g(x)|^2 \delta(x)^{1-2\eps} \diff x \leq C(\W,s)^2 \| g \|_{C^{\AJS{0,1-s}}(\overline \W)}^2 \int_\Omega \delta(x)^{-1+2(s-\eps)}\diff x,
\]
so that, by \cite[Remark 3.5]{AcostaBorthagaray} we obtain $w_g \in \tH^1_{1/2-\eps}(\Omega)$,
\(
  \AJS{| w_g |_{\tH^1_{1/2-\eps}(\Omega)}} \leq \frac{C(\W,s)}{\sqrt{\eps}} \| g \|_{C^{\AJS{0,1-s}}(\overline \W)}
\).

On the other hand, the last term in \eqref{eq:holder_linear}, and similar arguments to those elaborated in \cite[page 482]{AcostaBorthagaray} yield
\[
  \iint_{\W \times \W} \frac{ |\nabla w_g(x) - \nabla w_g(y)|^2}{|x-y|^{n+2s-4\varepsilon}} \delta(x,y)^{1-2\varepsilon} \diff x \diff y \leq \frac{C(\W,s)^2}{\eps^2} \| g \|_{C^{\AJS{0,1-s}}(\overline \W)}^2.
\]

It remains to treat a term of the form
\[
  \calI_\mathcal{O} = \int_\W |\nabla w_g(x) |^2 \int_{\cW} \frac1{|x-y|^{n+2s-4\eps}} \delta(x,y)^{1-2\eps} \diff y \diff x.
\]
Notice now that, for every $x \in \Omega$, integration in polar coordinates gives
\[
\int_{\cW} \frac1{|x-y|^{n+2s-4\eps}}  \diff y \le \int_{\compl{B(x,\delta(x))}} \frac1{|x-y|^{n+2s-4\eps}}  \diff y = \frac{\omega_{n-1}}{2(s-2\eps)} \, \delta(x)^{-2s+4\eps}.
\]
Therefore, we can bound
\[
  \calI_\mathcal{O} \leq C \int_\W |\nabla w_g(x)|^2 \delta(x)^{1-2\eps}\int_{\cW} \frac1{|x-y|^{n+2s-4\eps}}  \diff y \le C \int_\W |\nabla w_g(x)|^2 \delta(x)^{1-2s+2\eps}  \diff x,
\]
and because $\sup_{x \in \Omega} \delta(x)^{1-s} |\nabla w_g(x)| \le C \| g \|_{C^{\AJS{0,1-s}}(\overline \W)}$, we deduce that 
\[
\calI_\mathcal{O} \leq C \| g \|_{C^{\AJS{0,1-s}}(\overline \W)}^2 \int_\W \delta(x)^{-1+2\eps}  \diff x \le \frac{C}{\eps} \| g \|_{C^{\AJS{0,1-s}}(\overline \W)}^2,
\]
where we, again, used \cite[Remark 3.5]{AcostaBorthagaray} to bound the last integral.
\end{proof}

\subsection{The fractional obstacle problem: known results}
\label{sub:obstacle_known}

Let us now review the known results about the solution to the fractional obstacle problem \eqref{eq:obstacle}. First we remark that existence and uniqueness of a solution immediately follows from standard arguments, and that this solution is also the minimizer of the functional $\J$ over the set $\K$. Since this will be useful when dealing with approximation, it is now our intention to explore the equivalence of \eqref{eq:obstacle} with the complementarity system \eqref{eq:complementarity}. To do so, we first define the coincidence and non-coincidence sets, respectively, by
\[
  \Lambda = \left\{ x \in \W : u(x) = \chi(x) \right\}, \qquad N = \W \setminus \Lambda.
\]

\begin{proposition}[\eqref{eq:obstacle}$\implies$\eqref{eq:complementarity}]
\label{prop:equiv}
Let $\W$ be a bounded and Lipschitz domain that satisfies the exterior ball condition. Let $\chi \in C(\overline\W)$ satisfy $\chi \leq 0$ on $\partial\Omega$ and $f \in L^p(\Omega)$ for some $p>n/2s$. In this setting, the function $u \in \tHs$ that solves \eqref{eq:obstacle} satisfies $u \in C(\overline\W)$ as well as the complementarity conditions \eqref{eq:complementarity}.
\end{proposition}
\begin{proof}
Since $u \in \K$, then we have that $u-\chi \geq 0$ \mae $\Omega$. Let now $0 \leq \phii \in C_0^\infty(\Omega)$ and observe that the function $v = u + \phii \in \K$. This particular choice of test function in \eqref{eq:obstacle} implies that
\[
  (u,\phii)_s \geq \langle f, \phii\rangle
\]
and, using \eqref{eq:innerprodisduality} we conclude that
\[
  \left\langle \Laps u- f, \phii \right\rangle \geq 0, \quad \forall \phii \in C_0^\infty(\Omega), \ \phii \geq 0.
\]
In other words, $\lambda \geq 0$ in the sense of distributions.

On the other hand, according to \cite[Theorem 1.2]{MR3630405}, the assumptions imply that $u \in C(\overline\W)$ and, consequently, $N$ is an open set. Let $\phii \in C_0^\infty(N)$ and notice that, for a sufficiently small $\eps$ we have that $v = u \pm \eps \phii \in \K$. Using these test functions in \eqref{eq:obstacle} then implies that
\[
  \langle \lambda,\phii \rangle = 0, \quad \forall \phii \in C_0^\infty(N),
\]
as we intended to show.
\end{proof}

We will also make use of the following continuous dependence result.

\begin{lemma}[continuous dependence]
\label{lem:monotone}
Let $\chi \in L^\infty(\W)$, $f = 0$, and $u \in \tHs$ solve \eqref{eq:obstacle}. Then, we have that $u \in L^\infty(\W)$ with
\[
  \max\{\chi,0\} \leq u \leq \| \max\{\chi,0\} \|_{L^\infty(\W)} \ \mae \W.
\]
\end{lemma}
\begin{proof}
See \cite[Corollary 4.2]{MR3630405}.
\end{proof}

Below we will introduce further assumptions on the data $f$ and $\chi$ that will allow us to apply the previous results.

\section{Regularity}
\label{sec:regularity}

Having established the existence of solution and its equivalent characterization as the solution of \eqref{eq:complementarity}, we now begin with the study of its regularity. \AJS{To do so, we must introduce some notation. 
For 
a positive number $\kappa > 0$}
we \rhn{let $K_\kappa \in C^\infty(\rn)$ be a kernel so that}
\[
  K_\kappa(z) = \frac{C(n,s)}{|z|^{n+2s}}, \quad |z| \geq \kappa,
\]
and is extended smoothly for $|z|< \kappa$.

Finally, to concisely quantify the smoothness \AJS{assumptions on the right hand side $f$ and obstacle $\chi$ we introduce
\begin{equation}
\label{eq:defofcalF}
  \F(\overline\W) = \begin{dcases}
         C^{2,1-2s+\epsilon}(\overline\W), & s\in \left(0,\frac12\right], \\
         C^{1,2-2s+\epsilon}(\overline\W), & s \in \left(\frac12,1\right),
       \end{dcases}
  \qquad
  \calX(\overline\W) = \left\{ \chi \in C(\overline\W) : \chi_{|\pp\W} <0\right\}  \cap C^{2,1}(\W),
\end{equation}
where $\epsilon>0$ is sufficiently small, so that $1-2s+\epsilon$ is not an integer}.

\subsection{Interior regularity}
\label{sub:interior_reg}

The interior regularity of the solution to \eqref{eq:obstacle} will follow from the regularity for the case $\W = \rn$ as detailed in \cite{CaffarelliSalsaSilvestre}. Let us first slightly extend the main result in that work.

\begin{lemma}[regularity in $\rn$]
\label{lem:regRn}
Let $u \in \widetilde{H}^s(\rn)$ solve \eqref{eq:obstacle} with $\W=\rn$. If $\chi \in \AJS{\calX(\rn)}$, $f \in \F(\rn)$, and $f$ is such that $|f(x)|\leq C|x|^{-\sigma}$ for some $\sigma > 2s$ as $|x|\to \infty$, then we have $u \in C^{1,s}(\rn)$.
\end{lemma}
\begin{proof}
If $f = 0$, the assertion is the content of \cite[Corollary 6.10]{CaffarelliSalsaSilvestre}.
We now reduce the inhomogeneous case $f\ne0$ to the previous one by invoking the function $w_f$ defined, for $\W=\rn$, in \eqref{eq:Dirichlet}. Indeed, the function $U=u-w_f$ solves \eqref{eq:obstacle} with right hand side $f=0$ and obstacle $\chi-w_f$. Thus, to be able to invoke the reasoning for the homogeneous case, we must ensure that $\chi-w_f \in \AJS{\calX(\rn)}$. Since $\chi \in \AJS{\calX(\rn)}$ a sufficient condition for this, according to \cite[Propositions 2.8 and 2.9]{Silvestre}, is that $f \in \F(\rn)$ and $w_f \in L^\infty(\rn)$. To show the boundedness of $w_f$ we use its explicit representation
\[
  w_f(x) = C(n,-s) \int_\rn \frac{f(y)}{|x-y|^{n-2s}} \diff y;
\]
see \cite[Formula (1.1.12)]{Landkof} and \cite[Formula (2.3)]{Silvestre}.
Indeed, using the decay of $f$ we can estimate
\[
  |w_f(x)| \leq \| f \|_{L^\infty(B_R(x))} \int_{B_R}\frac1{|y|^{n-2s}} \diff y + C \int_{\compl{B_R }} \frac{|x+y|^{-\sigma}}{|y|^{n-2s}} \diff y \leq M.
\]
Since $w_f \in C^{3,\epsilon}(\rn) \subset C^{2,1}(\rn)$, we deduce $u = U+w_f \in C^{1,s}(\rn)$, and conclude the proof.
\end{proof}

With this result at hand we can establish the interior regularity of the solution to \eqref{eq:obstacle}. The idea is to use a direct localization argument. We point out that, for the fractional Laplacian a localization argument using the Caffarelli-Silvestre extension can be carried out, as described  in \cite[Section 2]{CaffarelliSalsaSilvestre}. Since for fractional Laplacians of order different than one half, the extension problem involves a degenerate elliptic equation with a weight that belongs to the Muckenhoupt class $A_2$ and depends only on the extended variable, this argument needs to combine fine estimates from \cite{FabesKenigSerapioni,MENGESHA2019184} with the translation invariance in the $x$-variable of the Caffarelli-Silvestre weight.

In this paper, instead, we pursue an entirely nonlocal approach. In that regard, the localization method we present in Proposition \ref{prop:localHolder} below can be applied also to more general nonlocal operators, such as those considered in \cite{CaffarelliRosOtonSerra}.
Finally, we stress that if $0\le\eta\le1$ is a smooth cut-off function such that $\eta=1$ in $\{ \chi > 0 \}$, then
\[
  \Laps (\eta u) \ne \eta \Laps u \quad \mbox{in } \{ \eta = 1 \}
\]
because of the nonlocal structure of $\Laps$. Consequently, we cannot deduce regularity of $\eta u$ directly from that of $\Laps u$. This is the difficulty we confront now.

\begin{remark}[Cauchy principal values]
\label{rem:valeur_principal}
At this point we must warn the reader about a technical aspect of our discussion. Namely, in what follows we will proceed formally and ``evaluate'' expressions of the form
\[
  \int_{\rn}\frac{w(y)}{|x-y|^{n+2s}} \diff y, \qquad \int_{\rn}\frac{w(x)-w(y)}{|x-y|^{n+2s}} \diff y,
\]
for some function $w:\rn \to \R$. Evidently, these integrals do not necessarily converge. We are doing this to avoid unnecessary technicalities, and what we mean in these cases is to compute the principal value of these integrals which, in the sense of distributions, is always meaningful. In other words, substitutions of the form
\[
  \int_{\rn}\frac{w(y)}{|x-y|^{n+2s}} \diff y \longleftrightarrow \lim_{\eps \downarrow 0} \int_{\rn \setminus B_\eps(x)}\frac{w(y)}{|x-y|^{n+2s}} \diff y
\]
need to be made below.
\eremk
\end{remark}

\AJS{As a final preparatory step we show that there is no loss of generality in assuming that the forcing term $f$ is zero, as the case $f\neq 0$ can be reduced to this one.

\begin{lemma}[reduction to $f = 0$]
\label{lem:reducefzero}
Let $f \in \F(\overline\W)$, $\chi \in \calX(\overline\W)$ and $u$ denote the solution to \eqref{eq:obstacle}. Then we have the representation
\[
  u = w_f + \widetilde{u},
\]
where $\widetilde{u}$ solves \eqref{eq:obstacle} with zero forcing ($f=0$) and obstacle $\widetilde \chi = \chi - w_f \in \calX(\overline\W)$.
\end{lemma}
\begin{proof}
It is well--known that by introducing the Lagrange multiplier $\lambda \in H^{-s}(\Omega)$, we have that
\[
  (-\Delta)^s u = f + \lambda,
\]
(cf. \eqref{eq:complementarity}) and, therefore, $u = w_f + w_\lambda$
Since $\lambda \geq 0$, using Proposition~\ref{prop:fractmaxprinciple} (nonlocal maximum principle) we infer that $w_\lambda \geq 0$ and, consequently, $u \geq w_f$.

Now, since $f \in \F(\overline\W)$, Proposition~\ref{prop:Holder} (H\"older estimates for the linear problem) gives, in particular, that $w_f \in C^{0,s}(\overline \W)$ and thus, since $w_f = 0$ on $\partial \Omega$ and $\chi \in \calX(\overline\W)$ there exists $r>0$ such that
\[
  x \in \Omega_r = \left\{ x \in \overline\W: \dist(x,\pp\W)< r \right\}
\]
implies $w_f(x) > \chi (x)$. Notice then that, \rhn{in $\Omega_r$,} we have $u \geq w_f > \chi$.

Define $\widetilde \chi = \chi - w_f \in C(\overline\W)$ and note that the previous considerations also give us that $\widetilde \chi<0$ on $\pp \W$. Moreover, since $f \in \F(\overline\W)$, the conclusion of \cite[Proposition 1.4]{RosOtonSerra} gives that $w_f \in C^{2,1}(\W)$. Thus, $\widetilde \chi \in \calX(\overline\W)$.

It remains now to realize that if we define $\widetilde u = u - w_f$, \rhn{then $\widetilde u$} solves the following version of \eqref{eq:obstacle}
\[
  \min\left\{ (-\Delta)^s \widetilde u,  \widetilde u - \widetilde\chi\right\} = 0,
\]
and that, $\widetilde{u} \geq 0 > \widetilde\chi$ on $\Omega_r$.
\end{proof}

Note that the usefulness of the previous result lies in the fact that, in our setting, the regularity of $u$ can be deduced from the regularity of the linear problem, which was described in Section~\ref{sub:linear}, and that of an obstacle problem without forcing and with an obstacle that has the same regularity of the original obstacle $\chi$.

Owing to the reduction given above, from now on we consider only the case $f = 0$.}

\begin{proposition}[interior H\"older regularity]
\label{prop:localHolder}
Let $\W$ be a bounded Lipschitz domain \rhn{and} \AJS{$\chi \in \calX(\overline\W)$. Then the solution $u \in \tHs$ of \eqref{eq:obstacle} with $f=0$} satisfies $u \in C^{1,s}(\W)$.
\end{proposition}
\begin{proof}
Let $D \Subset \W$ be open. Without loss of generality, we assume that $\{ \chi > 0 \} \Subset D$. Let, in addition, $\eta \in C_0^\infty(\W)$ be a smooth cutoff function such that
\[
  D \Subset \{ \eta \equiv 1 \}, \qquad \supp(\eta) \Subset \Omega, \qquad 0 \leq \eta \leq 1.
\]
Define $U = \eta u$. The objective is now to show that $U$ solves an obstacle problem with obstacle $\chi$ and a smooth right hand side $F$ with suitable decay at infinity, for if that is the case we can appeal to \AJS{Lemma~\ref{lem:regRn} (regularity in $\rn$)} to conclude that $U \in C^{1,s}(\rn)$. Since $U = u$ on $D$, the interior H\"older regularity of $u$ will follow.

We claim that $U \geq \chi$ in $\rn$. Indeed, if $\chi>0$ then $\eta = 1$ and $U=u\geq \chi$. On the other hand, since $0\leq \eta \leq 1$ we can multiply the inequality $u\geq \chi$ by $\eta$ to conclude that 
\[
  U = \eta u \geq \eta \chi,
\]
which, if $\chi \leq 0$, implies that $U \geq 0 \ge \chi$.

We now want to prove that there exists a smooth function $F$ such that $\Laps U \ge F$ and $\Laps U = F$ if $U > \chi$.
To accomplish this, we choose $\tau > 0$ sufficiently small so that
\[
  \dist\left( \supp(\eta), \pp\W \right) > 3 \tau, \qquad \dist\left( D, \pp\{\eta \equiv 1\} \right) > 3\tau.
\]
Finally, we let $\rn = A_1 \cup A_2 \cup A_3$, where
\begin{align*}
  A_1 &= \left\{ x \in \W:\ \dist(x,D) \leq 2\tau \right\}, \\
  A_2 &= \left\{ x \in \W:\ \dist(x,D)>\tau, \ \dist(x,\pp\W)>\tau  \right\}, \\
  A_3 &= \left\{ x \in \W:\ \dist(x,\pp\W) \leq 2\tau \right\} \cup \cW.
\end{align*}
We point out that this is not a partition, as the intersections $A_1 \cap A_2$ and $A_2 \cap A_3$ are not empty. However, since every point in $\rn$ is in the interior of at least one of the sets defined above, it suffices to show that the corresponding forcing is smooth on each of these sets separately.

\begin{enumerate}[$\bullet$]
  \item Let $x \in A_1$. Then, $\eta(y) = 1$ for all $y\in B_\tau(x)$ and we can write
  \[ \begin{aligned}
    \Laps U(x) & = \Laps u(x) + C(n,s)\int_{\rn} \frac{(1-\eta(y))u(y)}{|x-y|^{n+2s}} \diff y \\ 
   & = \Laps u(x) + \int_{\rn} \big[1-\eta(x-z)\big]u(x-z) K_\tau (z) \diff z,
  \end{aligned}\]
  where we use a fixed smooth kernel $K_\tau$ as described in the beginning of this section.
Using that $\Laps u(x) = \lambda (x) \geq 0$, we deduce that 
  \[
    \Laps U(x) \geq \left( K_{\tau} \star (1-\eta)u \right)(x).
  \]
  Moreover, if $x$ belongs to the non-coincidence set $N$, it follows that $\lambda(x) = 0$ and therefore
  \[
  \Laps U(x) = \left( K_{\tau} \star (1-\eta)u \right)(x).
  \]
  
  \item Given $x \in A_3$, we proceed essentially as before, except that now $\eta(y)= 0$ for all $y \in B_\tau(x)$, whence
  \[
    \Laps U(x) = C(n,s) \int_{\rn} \frac{-\eta(y)u(y)}{|x-y|^{n+2s}} \diff y = - \left( K_{\tau} \star \eta u\right)(x).
  \]
  
  \item Let $x \in A_2$. Since $\dist(x,\pp\W) > \tau$ and $\dist(x,\{\chi>0\})>\tau$, we deduce that $\chi < 0$ in $B_\tau(x)$ and, using Lemma~\ref{lem:monotone} (continuous dependence), that $u \in L^\infty(B_\tau(x))$ and
  \[
    \chi < 0 \leq u \ \mae B_\tau(x).
  \]
  From the complementarity conditions \eqref{eq:complementarity} it then follows that $u$ is $s$-harmonic in $B_\tau(x)$. Lemma~\ref{lemma:balayage} (balayage) implies that $u \in C^\infty(B_{\tau/2}(x))$, whence $U = \eta u \in C^\infty(B_{\tau/2}(x))$. Since this holds for every point in $A_2$, we deduce that $U \in C^\infty (A_2 + B_{\tau/2}(0))$, where $A_2 + B_{\tau/2}(0)$ is the Minkowski sum
  \[
  A_2 + B_{\tau/2}(0) = \{ x \in \rn \colon x = y + z, \quad y \in A_2, \ z \in  B_{\tau/2}(0) \}.
  \] 
  Let $\calE U \in C_0^\infty(\rn)$ be a smooth extension of $U$ outside $A_2 + B_{\tau/2}(0)$. Then, for $x \in A_2$ we can write    
 \[
    \Laps U(x) = C(n,s)\int_{\rn} \frac{\calE U(x) - U(y)}{|x-y|^{n+2s}} \diff y = \Laps \calE U(x) + C(n,s)\int_{\rn} \frac{\calE U(y) - U(y)}{|x-y|^{n+2s}} \diff y.
  \]
Since $\calE U(y) = U(y)$ for every $y \in B_{\tau/2} (x)$, we thus conclude
  \[
    \Laps U(x) = \Laps \calE U(x) + \left( K_{\tau/2} \star (\calE U - U) \right)(x).
  \]
\end{enumerate}

In order to gather the three cases considered above, we let 
\begin{equation} \label{eq:LapsUis}
F(x) = \begin{dcases}
                  \left( K_{\tau} \star (1-\eta)u \right)(x), & x \in A_1, \\
                  \Laps \calE (\eta u)(x) + \left( K_{\tau/2} \star (\calE (\eta u) - \eta u) \right)(x), & x \in A_2, \\
                  - \left( K_{\tau} \star \eta u\right)(x), & x \in A_3.
               \end{dcases}
\end{equation}
Observe that this expression is well defined. Indeed, on $A_2 \cap A_3$ both expressions coincide with $\Laps U$, while the fact that $A_1 \cap A_2 \subset N$ implies $\lambda=0$ and thus equality in this case. Therefore, we have defined a function $F\colon \rn \to \R$, which
is smooth in $\rn$ because so are the kernels $K_{\tau}, K_{\tau/2}$ and the function $\calE (\eta u)$.

From the considerations above, we deduce that $U$ solves an obstacle problem posed in $\rn$, with obstacle $\chi$ and right hand side $F$.
In addition, $u$ is bounded according to Lemma~\ref{lem:monotone}. We can thus derive for $x \in A_3$
\[
  |(K_{\tau} \star \eta u)(x)| \leq C \int_\Omega \frac1{|x-y|^{n+2s}} \diff y \leq C \dist(x,\Omega)^{-n-2s} \quad \mbox{as } |x|\to \infty,
\]
which gives the decay required in Lemma~\ref{lem:regRn} (regularity on $\rn$). As a consequence, we can then invoke Lemma~\ref{lem:regRn} to conclude that $U \in C^{1,s}(\rn)$. This in turn implies $u \in C^{1,s}(D)$ as asserted and concludes the proof.
\end{proof}

\begin{remark}[interior regularity estimate]
\label{rem:kappa}
Notice that, from \eqref{eq:LapsUis}, one can establish an estimate of $|u|_{C^{1,s}(\overline D)}$ in terms of $f$, $\chi$ and, more importantly $\tau$, which, essentially, measures how close the set $\{\chi >0\}$ is to the boundary $\pp\W$.
\eremk
\end{remark}

An immediate consequence of the interior H\"older regularity is an interior Sobolev regularity estimate.

\begin{corollary}[interior Sobolev regularity]
\label{col:locSobolev}
In the setting of Proposition~\ref{prop:localHolder} we have that, for every $\eps>0$, the solution $u$ to \eqref{eq:obstacle} satisfies $u \in H^{1+s-\eps}_{loc}(\W)$ with the estimate
\[
  |u|_{H^{1+s-\eps}(D)} \leq \frac{C(n)|D|^{1/2}\diam(D)^{\eps}}{\eps^{1/2}} |u|_{C^{1,s}(\overline D)},
\]
where $D\Subset \Omega$ is any open set and $\diam(D)$ denotes the diameter of $D$.
\end{corollary}
\begin{proof}
For $x,y \in D \Subset \W$ Proposition~\ref{prop:localHolder} (interior H\"older regularity) implies the bound 
\[
  |\nabla u(x) - \nabla u(y)| \leq |u|_{C^{1,s}(\overline D)} |x-y|^s.
\]
This bound, together with integration in polar coordinates, allow us to estimate directly the requisite seminorm as follows:
\begin{align*}
  |u|_{H^{1+s-\eps}(D)}^2 &= \iint_{D \times D} \frac{|\nabla u(x) - \nabla u(y)|^2}{|x-y|^{n+2s-2\eps} }\diff y \diff x
    \leq |u|_{C^{1,s}(\overline D)}^2 \iint_{D \times D} \frac1{|x-y|^{n-2\eps}} \diff y \diff x \\
    &\leq |u|_{C^{1,s}(\overline D)}^2 \omega_{n-1} |D| \int_0^{\diam(D)} \zeta^{-1+2\eps} \diff \zeta 
      = |u|_{C^{1,s}(\overline D)}^2 \omega_{n-1} |D| \frac{ \diam(D)^{2\eps} }{2\eps}.
\end{align*}
This is the asserted estimate.
\end{proof}

\subsection{Boundary regularity}
\label{sub:Bdry_reg}

Let us now study the behavior of the solution to \eqref{eq:obstacle} near the boundary of the domain $\pp\W$. It is here that the weighted Sobolev spaces introduced in Definition~\ref{def:wfSobolev} (weighted fractional Sobolev spaces) shall become important. We begin by recalling that we assume the obstacle $\chi$ to be a smooth function that is negative in a neighborhood of the boundary $\pp\W$. In other words, we have
\begin{equation}
\label{eq:defofvarrho}
  \varrho = \dist\left( \{ \chi>0\}, \pp \W \right) > 0.
\end{equation}
In the spirit to Remark~\ref{rem:kappa} (interior regularity estimate), the regularity estimates near the boundary will depend on $\varrho$. We now choose $\tau \in (0, \varrho/5)$ and define a boundary layer $\calB_\tau$ of width $\tau$, \ie
\begin{equation}
\label{eq:defoflayer}
  \calB_\tau = \left\{ x \in \overline\W : \dist(x,\pp\W) < \tau \right\}.
\end{equation}
Let $\eta \in C^\infty(\rn)$ be a smooth cutoff function such that
\[
  0 \leq \eta \leq 1, \quad \eta(x) = 1 \ \forall x \in \calB_{4\tau}, \quad \dist(\supp(\eta),\{\chi>0\}) > \tau.
\]
We finally set $N_\eta = \{\eta>0\}$.

Having introduced all the necessary notation, we proceed to establish the boundary regularity of $u$.

\begin{proposition}[boundary H\"older regularity of $\Laps (\eta u)$]
\label{prop:bdryHolder}
Let \AJS{$\chi \in \calX(\overline\W)$} and $f = 0$. With the notation introduced above, the function $\Laps (\eta u)$ is smooth in $N_\eta$. In particular, it holds that
\begin{equation} \label{eq:bdryHolder}
\|\Laps (\eta u) \|_{C^{\AJS{0,1-s}}(\overline{N_\eta})} \le C(\| u \|_{C^{1,s}(\W \setminus  \overline{\calB_{\tau}})}, \chi, s, n,\Omega, \varrho).
\end{equation}
\end{proposition}
\begin{proof}
We proceed as in the proof of Proposition~\ref{prop:localHolder} (interior H\"older regularity): we define $U= \eta u$, consider separately two overlapping sets $\calB_{3\tau}$ and $N_\eta\setminus\calB_{2\tau}$ and argue on each of these.

\begin{enumerate}[$\bullet$]
  \item Let $x \in \calB_{3\tau}$. Since $B_{\tau}(x) \cap \Omega\subset\calB_{4\tau}$, we have $\eta(y) = 1$ for all $y\in B_{\tau}(x)\cap \Omega$ and we can write
  \[
    \Laps U(x) = \Laps u(x) + C(n,s) \int_{\rn} \frac{(1-\eta(y))u(y)}{|x-y|^{n+2s}} \diff y.
  \]
We resort to Lemma~\ref{lem:monotone} (continuous dependence) once again to see that $u\ge 0 >\chi$ in $\calB_{3\tau}$, whence the complementarity conditions \eqref{eq:complementarity} imply that $\lambda(x) = \Laps u(x) = 0$. We deduce that
  \[
    \Laps U(x) = \left( K_{\tau} \star (1-\eta)u \right)(x) \quad \forall ~x\in \calB_{3\tau},
  \]
and therefore
  \[
    \|\Laps U\|_{C^{\AJS{0,1-s}}(\overline{\calB_{3\tau}})} \le \| K_{\tau} \|_{C^{\AJS{0,1-s}}(\rn)} \| (1-\eta)u \|_{L^1(\rn)} \leq C(\varrho, \chi, u).
  \]
  
\item Given $x \in N_\eta \setminus \calB_{2\tau}$, we still have that $\lambda =  \Laps u = 0$ in $B_\tau(x)$. Consequently, we can proceed as in the case $x \in A_2$ in the proof of Proposition~\ref{prop:localHolder} (interior H\"older regularity) to deduce that $\Laps U$ is smooth in $(N_\eta \setminus \calB_{2\tau}) + B_{\tau/2}(0)$. In fact, we construct a smooth extension $\calE U$ outside $(N_\eta \setminus \calB_{2\tau}) + B_{\tau/2}(0)$ that vanishes in $\compl{[(N_\eta \setminus \calB_{2\tau}) + B_{\tau}(0)]}$ to get
  \[
    \Laps U(x) = \Laps \calE U(x) + \left( K_{\tau/2} \star (\calE U - U) \right)(x) \quad\forall ~x \in N_\eta \setminus \calB_{2\tau},
  \]
 whence
  \begin{align*}
    \|\Laps U\|_{C^{\AJS{0,1-s}}(\overline{N_\eta \setminus \calB_{2\tau}})} \le \|\Laps \calE U\|_{C^{\AJS{0,1-s}}(\overline{N_\eta \setminus \calB_{2\tau}})} + 
      \| K_{\tau/2} \star (\calE U - U)\|_{C^{\AJS{0,1-s}}(\overline{N_\eta \setminus \calB_{2\tau}})}.
  \end{align*}
We next exploit that the extension satisfies $\calE U \equiv 0$ in  $\compl{[(N_\eta \setminus \calB_{2\tau}) + B_{\tau}(0)]}$, and so vanishes in $\calB_{\tau} \cup \compl{\Omega}$, to realize that
\[
 \|\Laps \calE U\|_{C^{\AJS{0,1-s}}(\overline{N_\eta \setminus \calB_{\tau}})}  \le C(s)  \| \calE U\|_{C^{1,s}(\rn)} \leq C(s) \| u \|_{C^{1,s}(\Omega \setminus \overline{\calB_{\tau}})}.
 \]
Moreover, since
\[
 \| K_{\tau/2} \star (\calE U - U)\|_{C^{\AJS{0,1-s}}(\overline{N_\eta \setminus \calB_{2\tau}})} \le C(\tau, \chi, u),
\]we deduce
\[
 \|\Laps U\|_{C^{\AJS{0,1-s}}(\overline{N_\eta \setminus \calB_{\tau}})} \le C(\| u \|_{C^{1,s}(\W \setminus \overline{\calB_{\tau}})}, \varrho, \chi).
 \]
\end{enumerate}
Combining the above H\"older estimates with the fact that $\calB_{3\tau}$ and $N_\eta\setminus\calB_{2\tau}$ overlap, \eqref{eq:bdryHolder} follows.
\end{proof}

The following simple argument reveals that the boundary behavior of $u$ coincides with that of linear equations. Let $g = \Laps (\eta u)$ and notice that, in $N_\eta$, the function $\eta u$ coincides with the solution $w_g$ of
\[
  \AJS{\Laps w_g = g, \ \text{in } N_\eta, \qquad w_g = 0, \ \text{in } \compl{N_\eta}}.
\]
We employ this relation to derive first a H\"older estimate and next a Sobolev estimate.  
We recall that $\varrho$ is defined in \eqref{eq:defofvarrho} and $\calB_\tau$ in \eqref{eq:defoflayer}.

\begin{corollary}[boundary H\"older regularity]
Let $\W$ be a bounded Lipschitz domain satisfying the exterior ball condition, and   
let $u \in \tHs$ solve \eqref{eq:obstacle} with \AJS{$\chi \in \calX(\overline\W)$} and $f = 0$. Then
\begin{equation} \label{eq:pointwise_bdry}
\|u\|_{C^\AJS{0,s}(\overline{N_\eta})} + 
\sup_{x \in \calB_{\tau}} \delta(x)^{1-s} |\nabla u(x)|\leq C(\| u \|_{C^{1,s}(\W \setminus  \overline{\calB_\tau})}, \chi, s, n,\Omega, \varrho).
\end{equation}
\end{corollary}
\begin{proof}
Since $g \in C^{\AJS{0,1-s}}(\overline{N_\eta})$ according to Proposition \ref{prop:bdryHolder} (boundary H\"older regularity), we can apply Proposition \ref{prop:Holder} (H\"older estimates for the linear problem) to deduce \eqref{eq:pointwise_bdry}.
\end{proof}

\begin{corollary}[boundary weighted Sobolev regularity]
\label{thm:bdyrreg}
Let $\W$ be a bounded Lipschitz domain satisfying the exterior ball condition, and let \AJS{$\chi \in \calX(\overline\W)$}, $f = 0$, and $u \in \tHs$ solve \eqref{eq:obstacle}. Then, for every \AJS{$\eps\in(0,s/2)$}, we have that $u \in H^{1+s-2\eps}_{1/2-\eps}(\calB_\tau)$ with the estimate
\begin{equation} \label{eq:boundary_1}
  |u|_{H^{1+s-2\eps}_{1/2-\eps}(\calB_\tau)} \leq \frac{C(\| u \|_{C^{1,s}(\W \setminus  \overline{\calB_\tau})}, \chi, s, n,\Omega, \varrho)}{\eps},
\end{equation}
where the weight $\delta$ refers to $\dist(\cdot, \pp\W)$. Moreover, we have the estimate
\begin{equation}
\label{eq:boundary_2}
  \iint_{\calB_\tau \times \cW} \frac{| \nabla u(x)|^2}{|x-y|^{n+2s - 4\eps}} \delta(x,y)^{1-2\eps} \diff y \diff x \le  
    \frac{C(\| u \|_{C^{1,s}(\W \setminus  \overline{\calB_\tau})}, \chi, s, n,\Omega, \varrho)}{\eps^2}.
\end{equation}
\end{corollary}
\begin{proof}
We apply Theorem~\ref{teo:weighted} (weighted regularity of $w_g$) to infer that $\eta u = w_g \in \tH^{1+s-2\eps}_{1/2-\eps}(N_\eta)$ with 
\[
\| \eta u \|_{\tH^{1+s-2\eps}_{1/2-\eps}(N_\eta)} \le \frac{C(\Omega,s)}{\eps} \|g\|_{C^{\AJS{0,1-s}}(\overline{N_\eta})} \le
\frac{C(\| u \|_{C^{1,s}(\W \setminus  \overline{\calB_\tau})}, \chi, s, n,\Omega, \varrho)}{\eps} .
\]
Notice that, in this estimate, the weight used to define the norm is the distance to $\pp N_\eta$. However, owing to the definition of $\calB_\tau$, we have that for all $x \in \calB_\tau$ this coincides with $\dist(x,\pp\W)$. In addition, since
$\eta \equiv 1$ on $\calB_\tau$ we can conclude that $u \in H^{1+s-2\eps}_{1/2-\eps}(\calB_\tau)$, with the corresponding estimate \eqref{eq:boundary_1}. 
Finally, recalling the definition of $\tH^{1+s-2\eps}_{1/2-\eps}(N_\eta)$ and restricting the integration to $\calB_\tau \times \cW$ instead of $\rn\times\rn$, the previous inequality yields \eqref{eq:boundary_2} because $u=0$ on $\compl{\Omega}$.
\end{proof}

\subsection{Global regularity}
\label{sub:globreg}
 
We are now in position to prove the global regularity of solutions of the fractional obstacle problem.

\begin{theorem}[global weighted Sobolev regularity for $f=0$] \label{thm:global_regularity}
Let $\W$ be a bounded Lipschitz domain satisfying the exterior ball condition, \AJS{$\chi \in \calX(\overline\W)$} satisfy \eqref{eq:defofvarrho}, and $f = 0$. Then, the solution $u \in \tHs$ of \eqref{eq:obstacle} satisfies $u \in \tH^{1+s-2\eps}_{1/2-\eps}(\W)$ for all \AJS{$\eps\in(0,s/2)$} with the estimate
\[
  |u|_{\tH^{1+s-2\eps}_{1/2-\eps}(\W)} \leq \frac{C(\chi, s, n,\Omega, \varrho)}{\eps},
\]
\AJS{\rhn{where} the constant in this estimate is independent of $\eps$}.
\end{theorem}
\begin{proof}
We split
\begin{equation} \label{eq:split_u}
|u|_{\widetilde H^{1+s-2\eps}_{1/2-\eps}(\W)}^2  = 
|u|_{H^{1+s-2\eps}_{1/2-\eps}(\W)}^2 + 2 \iint_{\W \times \cW} \frac{|\nabla u(x)|^2}{|x - y|^{n+2s-4\eps}} \delta(x,y)^{1-2\eps} \diff y \diff x ,
\end{equation}
and treat the two terms on the right hand side separately.
We bound the integral over $\W \times \W$ as follows:
\[
|u|_{H^{1+s-2\eps}_{1/2-\eps}(\W)}^2 \le |u|_{H^{1+s-2\eps}_{1/2-\eps}(\calB_{\tau})}^2 
+ |u|_{H^{1+s-2\eps}_{1/2-\eps}(\W\setminus \overline{\calB_{\tau/2}})}^2 + 2 \iint_{\calB_{\tau/2} \times (\W\setminus \overline{\calB_{\tau}})} \frac{|\nabla u(x)- \nabla u(y)|^2}{|x - y|^{n+2s-4\eps}} \delta(x,y)^{1-2\eps} \diff y \diff x .
\]
Theorem \ref{thm:bdyrreg} (boundary weighted Sobolev regularity) and Corollary \ref{col:locSobolev} (interior Sobolev regularity), respectively, give upper bounds for the first two terms on the right hand side. For the last term, we write 
\[ \begin{aligned}
\iint_{\calB_{\tau/2} \times (\W\setminus \overline{\calB_{\tau}})} \frac{|\nabla u(x)- \nabla u(y)|^2}{|x - y|^{n+2s-4\eps}} \delta(x,y)^{1-2\eps} \diff y \diff x \le &
2 \int_{\calB_{\tau/2}} |\nabla u(x)|^2 \delta(x)^{1-2\eps} \left( \int_{\W\setminus \overline{\calB_{\tau}}} \frac{1}{|x - y|^{n+2s-4\eps}}  \diff y \right) \diff x  \\ 
& + 2 \int_{\W\setminus \overline{\calB_{\tau}}} |\nabla u(y)|^2 \left( \int_{\calB_{\tau/2}} \frac{\delta(x)^{1-2\eps} }{|x - y|^{n+2s-4\eps}} \diff x \right) \diff y .
\end{aligned} \]
Since for every $(x,y) \in \calB_{\tau/2} \times (\W\setminus \overline{\calB_{\tau}})$ we have $|x-y| \ge \tau/2$, using the pointwise bound \eqref{eq:pointwise_bdry} and that owing to Remark \ref{rem:kappa} (interior regularity estimate) we have
\[
|\nabla u (y) | \le C(f, \chi, \tau), \quad y \in \W\setminus \overline{\calB_{\tau}},
\]
we conclude that the previous integral is independent of $\eps$ and
\[
|u|_{H^{1+s-2\eps}_{1/2-\eps}(\W)}^2 \le \frac{C(\chi, s, n,\Omega, \varrho)}{\eps}.
\]

We now consider the integral over $\W\times\cW$ in \eqref{eq:split_u}. In order to bound the integral over $\calB_\tau \times \cW$, we resort to estimate \eqref{eq:boundary_2}. On the other hand, if $x \in \W \setminus \calB_\tau$ and $y \in \cW$, then
$|x-y|>\tau$ yields
\[
\int_{\cW} \frac{1}{|x - y|^{n+2s-4\eps}} \diff y \le \omega_{n-1}\int_\tau^\infty t^{-1-2s+2\eps} \diff t \le \frac{\w_{n-1} \tau^{-2s+4\eps}}{2 (s-2\eps)},
\]
whereas $\delta(x,y) \le \delta(x) \le \diam(\W)$ implies
\[
\iint_{(\W\setminus\calB_\tau)\times\cW}  \frac{|\nabla u(x)|^2}{|x - y|^{n+2s-4\eps}} \delta(x,y)^{1-2\eps} \diff y \diff x \le C \int_{\W\setminus\calB_\tau} |\nabla u(x)|^2 \diff x \le C \|u\|_{C^{1,s}(\Omega\setminus\overline{\calB_\tau})},
\]
\AJS{
where 
\[
  C \leq \frac{\omega_{n-1} \tau^{-2s+4\eps}}{s(2-2\eps)} \diam(\W)^{1-2\eps},
\]
which can be bounded above independently of $\eps \in (0,s/2)$.}
Adding this estimate with \eqref{eq:boundary_2} finishes the proof.
\end{proof}

We conclude the discussion about the regularity of $u$ by treating the nonhomogeneous case $f \neq 0$.

\begin{corollary}[global weighted Sobolev regularity for $f\ne 0$]
\label{col:globregwithf}
Let $\W$ be a bounded Lipschitz domain satisfying the exterior ball condition, \AJS{$\chi \in \calX(\overline\W)$} satisfy \eqref{eq:defofvarrho}. Moreover, let \AJS{$ f \in \F(\overline\W)$} and $u \in \tHs$ be the solution to \eqref{eq:obstacle}. For every \AJS{$\eps \in(0,s/2)$} we have that $u \in \tH^{1+s-2\eps}_{1/2-\eps}(\Omega)$ with the estimate
\[
  |u|_{\tH^{1+s-2\eps}_{1/2-\eps}(\W)} \leq \frac{C(\chi, s, n, \Omega, \varrho, \| f \|_{\F(\overline\W)})}\eps.
\]
\end{corollary}
\begin{proof}
\AJS{Recall that, from Lemma~\ref{lem:reducefzero} (reduction to $f=0$) we have the representation
\[
  u = w_f + \widetilde{u}.
\]
Apply Theorem~\ref{teo:weighted} (weighted regularity of $w_g$) for $w_f$, and Theorem~\ref{thm:global_regularity} (global weighted Sobolev regularity for $f=0$) to $\widetilde{u}$ to \rhn{prove the asserted estimate.}}
%
\end{proof}

We conclude this section with a regularity result for $\lambda$ that will be useful in the sequel.

\begin{theorem}[H\"older regularity of $\lambda$]
\label{thm:reg_lambda}
Let $\lambda$ be defined in \eqref{eq:complementarity}. In the setting of Corollary~\ref{col:globregwithf} we have that $\lambda \in C^{\AJS{0,1-s}}(\overline\W)$.
\end{theorem}
\begin{proof}
  We begin by observing that $\varrho>0$ according to \eqref{eq:defofvarrho} and the coincidence set $\Lambda \Subset \Omega$. Consequently $\lambda \equiv 0$ in the non-coincidence set $N$ and we need to prove the asserted regularity of $\lambda$ in $\Lambda$.

The arguments below mimic ideas used to prove Proposition~\ref{prop:localHolder} (interior H\"older regularity). We introduce a smooth cutoff function $\eta$ such that $\eta \equiv 1$ on $\Lambda$ and, for some $\tau >0$,
\[
  \dist(\supp(\eta), \pp\W) > 2\tau, \qquad \dist(\Lambda, \pp\{\eta = 1 \} ) > 2\tau.
\]
Define now 
\[
  \Lambda_\tau = \left\{ x \in \Omega : \dist(x,\Lambda)\leq \tau \right\}
\]
and let $x \in \Lambda_\tau$. Since $\eta \equiv 1$ on $B_\tau(z)$, we are now in a similar situation to the case $z \in A_1$ in the proof of Proposition~\ref{prop:localHolder}. Then we have for all $x\in B_\tau(z)$
\[
  \Laps u(x) = \Laps(\eta u)(x) - C(n,s) \int_{\rn} \frac{(1-\eta(y))u(y)}{|x-y|^{n+2s}} \diff y = \Laps(\eta u)(x) - \left( K_{\tau/2} \star (1-\eta)u \right)(x),
\]
where the last identity holds because $\eta(y)=1$ for $|z-x|, |x-y| \le\tau/2$. Since $u \in C^{1,s}(\W)$ and $\eta$ is smooth, we deduce that the first term $\eta u \in C^{1,s}(\rn)$ and $\Laps (\eta u) \in C^{\AJS{0,1-s}}(\rn)$. On the other hand, the second term $K_{\tau/2} \star (1-\eta)u$ is smooth in $B_\tau(z)$ which in turn is arbitrary. This implies $\Laps u \in C^{\AJS{0,1-s}}(\Lambda_\tau)$.

Finally, since $f \in \F(\overline\W) \subset C^{\AJS{0,1-s}}(\overline\W)$ we conclude that $\lambda = \Laps u - f \in C^{\AJS{0,1-s}}(\overline\Lambda_\tau)$.
\end{proof}

\section{Finite element approximation}
\label{sec:FEM}

In this section we will apply regularity estimates in weighted Sobolev spaces shown in Section~\ref{sec:regularity} to derive near optimal rates of convergence for a finite element method (FEM) for \eqref{eq:obstacle} over graded meshes. The latter compensate for the singular boundary behavior of solutions regardless of domain smoothness, which is a distinctive feature of fractional diffusion problems for any fractional order $s\in (0,1)$.

Let us then begin by describing the discrete framework that we will adopt. First, to avoid technicalities we shall assume, from now on, that $\W$ is a polytope and so convex owing to the exterior ball condition. Next, we introduce a family $\{\T\}_{h>0}$ of conforming and simplicial triangulations of $\overline\W$ which we assume shape regular, \ie we have that
\[
  \sigma = \sup_{h>0} \sup_{T \in \T} \frac{h_T}{\rho_T} <\infty,
\]
where $h_T =\diam(T)$ and $\rho_T $ is the diameter of the largest ball contained in $T$. The vertices of $\T$ will be denoted by $\{\x_i\}$. We comment that we assume that the elements $T \in \T$ are closed. In this case the star, patch, or first ring of $T \in \T$ is defined as
\[
  S^1_T = \bigcup \left\{ T' \in \T: T \cap T' \neq \emptyset \right\}.
\]
We also introduce the star of $S^1_T$ (or second ring of $T$),
\[
  S^2_T = \bigcup \left\{ T' \in \T: S^1_T \cap T' \neq \emptyset \right\}.
\]
Below, when discussing positivity preserving interpolation over fractional order smoothness spaces we partition $\T$ into two classes, interior and boundary elements, as follows:
\begin{equation}
\label{eq:defofinteriorandaway}
\calT_h^\circ = \left\{ T \in \T: S^1_T \cap \pp \W = \emptyset \right\}, \qquad
\calT_h^\pp = \left\{ T \in \T: S^1_T \cap \pp \W \neq \emptyset \right\}.
\end{equation}

On the basis of the triangulation $\T$ we define $V_h$ as the space of continuous, piecewise affine functions on $\T$ that vanish on $\pp\W$. The Lagrange nodal basis of $V_h$ will be denoted by $\{\phii_i\}$ and
\[
  S_i = \supp (\phii_i) .
\]
We will denote by $B_i$ the maximal ball, centered at $\x_i$, and contained in $S_i$. If $\rho_i$ is the radius of $B_i$, and $h_i = \diam(S_i)$ by shape regularity of the mesh we have the equivalences $\rho_i \approx h_i \approx h_T$, for all $T \subset S_i$.

\subsection{Positivity preserving interpolation over fractional order spaces}
\label{sub:positive}

Below it will become necessary to introduce a discrete version of the admissible set $\K$ defined in \eqref{eq:def_K}. In addition, when performing the analysis of the FEM it will become necessary that an interpolator of the exact solution belongs to this discrete admissible set. Since we assume that \AJS{$\chi \in \calX(\overline\W)$} and $f \in \F(\overline\W)$, we have that $u \in C(\overline\W)$ as a consequence of Proposition~\ref{prop:equiv} (\eqref{eq:obstacle}$\implies$\eqref{eq:complementarity}). Therefore one could, in principle, use the Lagrange interpolation operator. It turns out, however, that this operator does not possess suitable stability and approximation properties with respect to fractional order Sobolev spaces. For this reason, we will use instead the operator $I_h$ introduced in \cite{ChenNochetto} which we now describe.

\begin{definition}[positivity preserving interpolation operator]
\label{def:CZinterpolant}
Let $I_h : L^1(\W) \to V_h$ be defined by
\[
  I_h v = \sum_{i \colon \x_i \in \Omega} \left( \frac1{|B_i|} \int_{B_i} v(x) \diff x \right) \phii_i.
\]
\end{definition}

Notice that, since the sum is only over interior vertices of $\T$, we indeed have that $I_h v$ vanishes on $\pp\W$, whence $I_hv \in V_h$. In addition, by construction, this operator is positivity preserving: we have that $I_h v \geq 0$ whenever $v\geq 0$. Moreover, since for every $\x_i \in \Omega$ the ball $B_i$ is symmetric with respect to $\x_i$ we have the following exactness property for $I_h$
\begin{equation}
\label{eq:linear}
  I_h v(\x_i) = v(\x_i), \quad \forall v \in \polP_1(B_i),
\end{equation}
where by $\polP_1(E)$ we denote the space of polynomials of degree one over the set $E$. Notice however, that this operator is not a projection. In general, if $v_h \in V_h$ then $I_h v_h \neq v_h$; see \cite{NochettoWahlbin} for details. The following result summarizes the local stability and approximation properties of $I_h$.

\begin{proposition}[properties of $I_h$]
\label{prop:propsIh}
Let $p \in [1,\infty]$, $I_h$ be the operator introduced in \AJS{Definition~\ref{def:CZinterpolant} (positivity preserving interpolation operator)}, and $T \in \T$. Then, there are constants independent of $T$ and $h$ such that
\[
  \| I_h v \|_{L^p(T)} \leq C\| v \|_{L^p(S^1_T)}, \quad \forall v \in L^p(\Omega),
\]
and
\[
  \| \nabla I_h v \|_{L^p(T)} \leq C \| \nabla v \|_{L^p(S^1_T)}, \quad \forall v \in W^{1,p}_0(\Omega).
\]
Moreover, for \AJS{$t \in [1,2]$}, we also have the error estimate
\[
  \AJS{\| v - I_h v \|_{L^p(T)} \leq C h_T^{t} |v|_{W^{t,p}(S^1_T)}, \quad \forall v\in W^{t,p}(\Omega) \cap W^{1,p}_0(\Omega)}.
\]
\end{proposition}
\begin{proof}
See \cite[Lemmas 3.1 and 3.2]{ChenNochetto}. The fractional error estimates follows from interpolation between the cases \rhn{$t=1$ and $t=2$} in \cite[Lemma 3.2]{ChenNochetto}.
\end{proof}

We need to obtain similar properties in fractional order Sobolev spaces, and for that we will follow the ideas of \cite{Ciarlet}. We begin with a local stability estimate over the set $T\times S^1_T$, which exhibits the least amount of overlap for every $T\in\T$ to control the \emph{nonlocal} fractional Sobolev norms \cite{Faermann2,Faermann}.

\begin{proposition}[local stability of $I_h$]
\label{prop:stabIh}
Let $s \in (0,1)$ and $T \in \T$. There is a constant $C(n,\sigma)$, depending only on the dimension $n$ and the shape regularity parameter $\sigma$ of the mesh, such that the estimate
\[
  \iint_{T \times S^1_T} \frac{ |I_hv(x) - I_hv(y) |^2}{|x-y|^{n+2s}} \diff y \diff x \leq \frac{C(n,\sigma)}{1-s} h_T^{n-2s} \sum_{i: \x_i \in S^1_T} \left( \frac1{|B_i|} \int_{B_i} v(z) \diff z \right)^2
\]
holds for all $v \in L^1(\Omega)$.
\end{proposition}
\begin{proof}
From Definition~\ref{def:CZinterpolant} (positivity preserving interpolation operator) it follows that, if $x \in T$ and $y \in S^1_T$, then
\[
  I_h v(x) - I_h v(y) = \sum_{i:\x_i \in S^1_T} \left( \frac1{|B_i|} \int_{B_i} v(z) \diff z \right) ( \phii_i(x) - \phii_i(y)).
\]
In addition we observe that, by shape regularity the number of terms in this sum is uniformly bounded by a constant that depends only on $\sigma$. Thus, by H\"older's inequality we have that
\begin{multline*}
  \iint_{T \times S^1_T} \frac{ |I_hv(x) - I_hv(y) |^2}{|x-y|^{n+2s}} \diff y \diff x \leq C(\sigma) \sum_{i:\x_i \in S^1_T} \left( \frac1{|B_i|} \int_{B_i} v(z) \diff z \right)^2 
    \iint_{T \times S^1_T} \frac{ |\phii_i(x) - \phii_i(y) |^2}{|x-y|^{n+2s}} \diff y \diff x.
\end{multline*}
From mesh regularity it follows that $|\phii_i|_{C^{0,1}(\overline\W)} \leq C(\sigma)h_T^{-1}$ uniformly in $i$ and that
\[
  \alpha(x) = \max_{z \in S^1_T} |x-z| \leq C(\sigma) h_T.
\]
These two observations and integration in polar coordinates then imply that
\[
  \iint_{T \times S^1_T} \frac{ |\phii_i(x) - \phii_i(y) |^2}{|x-y|^{n+2s}} \diff y \diff x \leq \frac{C(\sigma)}{h_T^2} \iint_{T \times S^1_T} |x-y|^{2-n-2s} \diff y \diff x
    \leq \frac{C(n,\sigma)}{h_T^2} \int_T \int_0^{\alpha(x)} \rho^{1-2s} \diff \rho \diff x.
\]
From this the asserted estimate immediately follows.
\end{proof}

Let now $S\subset \rn$. It is well-known that for every $v \in W^{k,1}(S)$ there is a unique polynomial $P_kv$ of degree $k$ that satisfies
\begin{equation}
\label{eq:averaging}
  \int_S \partial^\alpha(v-P_kv) \diff x = 0, \quad \forall \alpha \in\N^n,\  |\alpha|\leq k.
\end{equation}
We shall also need the following fractional Poincar\'e inequality.

\begin{proposition}[fractional Poincar\'e inequality]
Let $s \in (0,1)$, $\alpha \in [0,s)$ and $S$ be a domain which is a finite union of overlapping star-shaped domains $S_i$ with respect to balls $B_i$, $i=1,\ldots,I$. Then, there exists a constant $C>0$, depending on the chunkiness of $S_i$ and the amount of overlap between the subdomains $S_i$, such that, for any $i \in \{1,\ldots,I\}$, we have
\begin{equation}
\label{eq:poincare}
  \| v - \overline{v}_i \|_{L^2(S)} \le C \diam(S)^{s-\alpha} |v|_{H^s_\alpha(S)}, \quad \forall v \in H^s_\alpha(S),
\end{equation}
where $\overline{v}_i = \tfrac1{|S_i|}\int_{S_i} v(x) \diff x$.
\end{proposition}
\begin{proof}
We must first observe that when $S$ is itself star-shaped, the result is proved in \cite[Proposition 4.8]{AcostaBorthagaray}.

In the general case, the result is an easy modification of the arguments used to show \cite[Theorem 7.1]{DupontScott}; see also \cite[Corollary 3.2]{nochetto2016piecewise} and \cite[Corollary 4.4]{NochettoOtarola}. For brevity we skip the details.
\end{proof}

Notice that, as a consequence of the fractional Poincar\'e inequality \eqref{eq:poincare}, we have that, whenever $t \in (1,2)$ and $\alpha \in [0, t-1)$, there are constants that depend only on $\sigma$ such that, for every $v \in H^t_\alpha(S^2_T)$, the polynomial $P_1v$, defined by \eqref{eq:averaging} with $S=S_T^2$, satisfies
\begin{align*}
  \| v - P_1 v \|_{L^2(S^2_T)} &\leq C h_T^{t-\alpha} | v |_{H^t_{\alpha}(S^2_T)}, \\
  \| \nabla(v - P_1 v) \|_{L^2(S^2_T)} &\leq C h_T^{t-\alpha-1} | v |_{H^t_\alpha}(S^2_T).
\end{align*}
\rhn{We use \cite[Lemma 23.1]{Tartar07} to interpolate these two inequalities and obtain that, whenever $s \in [0,1]$, $t \in(1,2)$, and $\alpha \in[0,t-1)$, there is a constant $C$} that depends only on $\sigma$ for which
\begin{equation}
\label{eq:est_P}
  |v-P_1v|_{H^s(S^2_T)} \leq C h_T^{t-\alpha-s} |v|_{H^t_\alpha(S^2_T)}.
\end{equation}

With these estimates at hand, we now proceed to obtain local interpolation error estimates for $I_h$ of Definition~\ref{def:CZinterpolant} (positivity preserving interpolation operator). We must do this separately for interior and boundary elements, as defined in \eqref{eq:defofinteriorandaway}. We first give the interior estimate and next the boundary estimate.

\begin{proposition}[interior interpolation estimate]
\label{prop:interpolation_interior}
Let $\calT_h^\circ$ be defined in \eqref{eq:defofinteriorandaway} and $T \in \calT_h^\circ$. Assume, in addition, that $s \in (0,1)$, $t \in (1,2)$, and that $I_h$ is the positivity preserving interpolator of Definition~\ref{def:CZinterpolant}. Then, there is a constant $C(n,\sigma,t)$ that depends only on the dimension $n$, the shape regularity parameter $\sigma$, and $t$ such that
\[
  \iint_{T \times S^1_T} \frac{|(v-I_h v) (x) - (v-I_h v) (y)|^2}{|x-y|^{n+2s}} \diff y \diff x \le \frac{C(n,\sigma, t)}{1-s} h_T^{2(t-s)} |v|_{H^t(S^2_T)}^2,
\]
\AJS{where the constant $C(n,\sigma, t)$ is \rhn{non-decreasing} in $t$}.
\end{proposition}
\begin{proof}
We begin by writing $v - I_hv = (v-P_1v) + (P_1v-I_hv)$, where $P_1v \in \polP_1$ is the polynomial defined by \eqref{eq:averaging} over $S^2_T$. We estimate the two terms on the right hand side separately.

Using \eqref{eq:est_P} with $\alpha = 0$ the first term can be estimated as follows:
\[
  \iint_{T \times S^1_T} \frac{|(v-P_1 v) (x) - (v-P_1 v) (y)|^2}{|x-y|^{n+2s}} \diff y \diff x \leq |v-P_1 v|^2_{H^s(S^1_T)} \le Ch_T^{2(t-s)}|v|^2_{H^t({S^2_T})}.
\]

On the other hand, since $P_1v \in \polP_1(S^2_T)$ it follows, from \eqref{eq:linear}, that $I_hP_1v_{|S^1_T} = P_1v_{|S^1_T}$ and to control the second term we only need to invoke Proposition~\ref{prop:stabIh} (local stability of $I_h$) to arrive at
\begin{align*}
  \iint_{T \times S^1_T} \frac{|(P_1 v - I_h v) (x) - (P_1 v - I_h v)(y)|^2}{|x-y|^{n+2s}} \diff y \diff x &\leq 
    \frac{C(n,\sigma)}{1-s} h_T^{n-2s} \sum_{i \colon \x_i\in S^1_T} \frac{1}{|B_i|}\| v - P_1 v\|_{L^2(B_i)}^2 \\
    &\leq \frac{C(n,\sigma)}{1-s} h_T^{-2s} \| v - P_1 v\|_{L^2(S^2_T)}^2.
\end{align*}
Setting $s=\alpha=0$ in \eqref{eq:est_P} yields the desired estimate.
\end{proof}

As a final preparatory step we obtain local interpolation error estimates for elements in $\calT_h^\pp$

\begin{proposition}[boundary interpolation estimate]
\label{prop:interpolation_boundary}
Let $\calT_h^\pp$ be defined in \eqref{eq:defofinteriorandaway} and $T \in \calT_h^\pp$. Assume, in addition, that $s \in (0,1)$, $t \in (1,2)$, $\alpha \in [0,1/2)$, and that $I_h$ is the positivity preserving interpolation operator of Definition~\ref{def:CZinterpolant}. Then, there is a constant $C(n,\sigma,t)$ that depends only on the dimension $n$, the shape regularity parameter $\sigma$, and $t$ such that, for all $v \in \tH_\alpha^t(\W)$, we have
\[
  \iint_{T \times S^1_T} \frac{|(v-I_h v) (x) - (v-I_h v) (y)|^2}{|x-y|^{n+2s}} \diff y \diff x \leq \frac{C(n,\sigma, t)}{1-s} h_T^{2(t-s-\alpha)} |v|_{H_\alpha^t(S^2_T)}^2,
\]
\AJS{where the constant $C(n,\sigma,t)$ is \rhn{non-decreasing} in $t$}.
\end{proposition}
\begin{proof}
As in the proof of Proposition~\ref{prop:interpolation_interior} (interior interpolation estimate) we decompose $v - I_hv = (v-P_1v) + (P_1v-I_hv)$ and estimate each term separately. For the first term, we use \eqref{eq:est_P} to obtain
\[
  \iint_{T \times S^1_T} \frac{|(v-P_1 v) (x) - (v-P_1 v) (y)|^2}{|x-y|^{n+2s}} \diff y \diff x  \le Ch_T^{2(t-s-\alpha)}|v|^2_{H^t_\alpha({S^2_T})}.
\]

The estimate of the second term $P_1v - I_h v$ is now more delicate, as we cannot exploit the symmetries that $T \in \calT_h^\circ$ afforded us in Proposition~\ref{prop:interpolation_interior} (interior interpolation estimate). Instead, we will follow the ideas used to obtain \cite[Lemma 3.2]{ChenNochetto}, where a similar difficulty is handled by further decomposing this term into
\[
  P_1 v - I_h v = I_h(P_1 v - v) + (P_1v - I_hP_1v).
\]
Proposition~\ref{prop:stabIh} (local stability of $I_h$) and estimate \eqref{eq:est_P} for $s=0$ allow us to bound the first term:
\[
  \iint_{T \times S^1_T} \frac{|I_h(P_1 v- v) (x) - I_h(P_1 v- v) (y)|^2}{|x-y|^{n+2s}} \diff y \diff x  \leq \frac{C(n,\sigma, t)}{1-s} h_T^{2(t-s-\alpha)} 
    |v|_{H_\alpha^t(S^2_T)}^2.
\]
Next, we notice that the difference $P_1v - I_hP_1v$ can be written, for $x \in S^1_T$, as
\[
  (P_1v - I_hP_1v)(x) = \sum_{j \colon \x_j\in S^1_T} \left( P_1 v (\x_j) - I_h P_1 v (\x_j) \right) \phii_j (x);
\]
where now the summation must include the vertices $\x_j \in S^1_T \cap \pp\W$, where $I_hP_1v(\x_j) = 0$ but $P_1v(\x_j) \neq 0$ in general. Since, by shape regularity, the number of indices in this sum is uniformly bounded and $0 \leq \phii_j \leq 1$, we can proceed as in Proposition~\ref{prop:stabIh} to obtain
\[
  \iint_{T \times S^1_T} \frac{| (P_1 v  - I_h P_1 v) (x) - (P_1 v  - I_h P_1 v) (y)|^2}{|x-y|^{n+2s}} \diff y \diff x  
    \le \frac{C(n,\sigma)}{1-s} h_T^{n-2s} \sum_{j \colon \x_j\in S^1_T} \left( (P_1 v - I_h P_1 v) (\x_j) \right)^2.
\]
The objective is now to show that, for all indices in the indicated range, 
\[
  \left( (P_1 v - I_h P_1 v) (\x_j) \right)^2 \leq C h_T^{-n+2(t-\alpha)}|v|_{H^t_\alpha(S^2_T)}^2,
\]
as this will imply the desired estimate. If $\x_j \in \W$ then we get
\[
  (P_1 v - I_h P_1 v) (\x_j) = 0,
\]
in view of \eqref{eq:linear}.
On the other hand if $\x_j \in \pp \W$, then $I_hP_1v(\x_j) = 0$. Let $\x_j \in e_j \subset \pp\W \cap S^1_T$ be a face and recall the scaled trace inequality
\[
  \| w \|_{L^2(e)} \leq C \left( h_e^{-1/2} \| w \|_{L^2(T)} + h_e^{1/2} \| \nabla w \|_{L^2(T)} \right) \quad \forall w\in H^1(T).
\]
This, for $w=v-P_1v$, together with an inverse inequality and the fact that $v|_{e_j} = 0$, yields
\begin{align*}
    |P_1 v (\x_j)| &\leq C h_T^{(1-n)/2} \| P_1 v \|_{L^2(e_j)} = C h_T^{(1-n)/2} \| P_1 v - v\|_{L^2(e_j)} \\
      &\leq C h_T^{(1-n)/2} \left( h_T^{-1/2} \|v - P_1v \|_{L^2(T)} + h_T^{1/2} \| \nabla (v - P_1 v) \|_{L^2(T)} \right).
\end{align*}
Property $v|_{e_j} = 0$ is a consequence of \cite[Theorem 2.3]{MR1258430}, because $v \in \tH^t_\alpha(\W) \subset \tH^1_\alpha(\W)$.
An application of \eqref{eq:est_P} for $s =0$ and $s=1$ allows us to conclude the proof.
\end{proof}

\begin{remark}[case $s=0$]
\label{rem:s_equals_0}
We briefly comment that Proposition~\ref{prop:interpolation_boundary} (boundary interpolation estimate)  can be extended to $s=0$. In fact, if $T \in \calT_h^\pp$, and $t$ and $\alpha$ are as in Proposition~\ref{prop:interpolation_boundary}, then we have
\[
  \| v - I_h v \|_{L^2(T)} \leq C h_T^{t - \alpha} |v|_{H_\alpha^t(S^2_T)},
\]
for every $v \in \widetilde{H}_\alpha^t(\Omega)$. The proof is a slight modification of the arguments needed for $s>0$ and, for brevity, we skip the details.
\eremk
\end{remark}

We are now finally in position to prove global interpolation error estimates. While Propositions~\ref{prop:interpolation_interior} (interior interpolation estimate) and \ref{prop:interpolation_boundary} (boundary interpolation estimate) may allow us to obtain error estimates over quasi-uniform meshes for functions in $H^t(\Omega)$, $t \in (1,2)$, the regularity results of Section~\ref{sec:regularity} show that these may be of little use for the approximation of problem \eqref{eq:obstacle}. We will, instead, exploit the regularity estimates in weighted Sobolev spaces $H^t_\alpha(\Omega)$ of Section \ref{sec:regularity} in conjunction with mesh grading towards the boundary to compensate for the singular behavior of the solution.

\AJS{The \rhn{preceding discussion motivates the use of graded meshes. In addition, these meshes must be shape regular for Propositions \ref{prop:interpolation_interior} and \ref{prop:interpolation_boundary} to hold.} For these reasons the meshes $\T$ that we consider will be constructed as follows.} Given a mesh parameter $h>0$ and $\mu \in [1,2]$ every element $T \in \T$ satisfies
\begin{equation}
\label{eq:mesh_grad}
  \begin{dcases}
    h_T \approx C(\sigma) h^\mu, & T \in \calT_h^\pp \\ 
    h_T \approx C(\sigma) h \dist(T,\pp\W)^{(\mu-1)/\mu}, & T \in \calT_h^\circ.
  \end{dcases}
\end{equation}

\begin{remark}[dimension of $V_h$]
\label{rem:cardinality}
\AJS{Following \cite[Lemma 4.1]{BKP:79} it is not difficult to see that the space $V_h$ constructed over the mesh $\T$ that satisfies \eqref{eq:mesh_grad} will satisfy 
\[
  \dim V_h \approx \begin{dcases}
                     h^{(1-n)\mu}, & \mu > \frac{n}{n-1}, \\
                     h^{-n}|\log h|, & \mu  = \frac{n}{n-1}, \\
                     h^{-n}, & \mu < \frac{n}{n-1}.
                   \end{dcases}
\]
Indeed, since 
the mesh is assumed shape regular, we have that
\[
  \dim V_h \leq (n+1) \sum_{T \in \T} 1 
    \leq C(\sigma) \left(  \sum_{T \in \calT_h^\circ} h_T^{-n}\int_T \diff x + \sum_{T \in \calT_h^\pp} h_T^{-n} \int_T \diff x \right).
\]
Over $\calT_h^\pp$, because $\cup_{T \in \calT_h^\pp} T$ defines a layer around the boundary of thickness about $h^\mu$, we have
\[
  \sum_{T \in \calT_h^\pp} h_T^{-n} \int_T \diff x \leq C h^{-n\mu} \sum_{T \in \calT_h^\pp} \int_T \diff x \leq C h^{(1-n)\mu}.
\]
On the other hand, for $\calT_h^\circ$ we have
\[
  \sum_{T \in \calT_h^\circ} h_T^{-n}\int_T \diff x \leq C h^{-n} \int_{h^\mu}^{\diam(\W)} \rho^{-\AJS{n}(\mu-1)/\mu} \diff \rho = 
                   \begin{dcases}
                     h^{(1-n)\mu}, & \mu > \frac{n}{n-1}, \\
                     h^{-n}|\log h|, & \mu  = \frac{n}{n-1}, \\
                     h^{-n}, & \mu < \frac{n}{n-1}.
                   \end{dcases}
\]
In other words, if we wish that the dimension of $V_h$ scaled like (up to logarithmic factors) $h^{-n}$ we must set the grading to be $\mu\leq n/(n-1)$.

For future reference we record that, if we insist on setting $\mu = 2$, then we obtain
\[
  \dim V_h = \begin{dcases}
                     h^{-2}|\log h|, &  n=2, \\
                     h^{-4}, & n=3.
                   \end{dcases}
\]
In three dimensions $\mu=2$ does not yield an optimal number of degrees of freedom.
}
\eremk
\end{remark}

Before we proceed further, we present the following inequality regarding the localization of fractional order Sobolev seminorms, and refer the reader to \cite{Faermann,Faermann2} for a proof:
\begin{equation}
\label{eq:faermann_new}
  |v|_{H^s(\W)}^2 \leq \sum_{T \in \T} \left[ \iint_{T \times S^1_T} \frac{|v (x) - v (y)|^2}{|x-y|^{n+2s}} \diff y \diff x 
    + \frac{2 \w_{n-1}}{s h_T^{2s}} \, \| v \|^2_{L^2(T)} \right].
\end{equation}

Let us now show a global interpolation estimate for functions in $\tH^{1+s-2\eps}_{1/2-\eps}(\Omega)$, in two dimensions, over graded meshes that satisfy \eqref{eq:mesh_grad}.

\begin{theorem}[global interpolation estimate]
\label{thm:interpolation}
\AJS{Let $\T$ be shape regular and satisfy the mesh grading condition \eqref{eq:mesh_grad} with $\mu\in[1,2]$. Assume, in addition, that $t \in (1,2)$ and $\eps\in(0,1/4)$. Define
\[
  \alpha = \begin{dcases}
             \left(\tfrac{\mu-1}\mu\right)(t - s), & s \neq \frac12, \\
             \left(\tfrac{\mu-1}\mu\right)\left(t - \frac12 - \eps \right), & s = \frac12.
           \end{dcases}
\]
Then, there is a constant $C$ that depends only on $s$, $\Omega$ and $\sigma$ such that,
\begin{equation} \begin{aligned}
\label{eq:interpolacion}
 & | v - I_h v |_{\tHs} \leq C h^{t-s} |v|_{\tH^{t}_{\alpha}(\Omega)} & (s \ne 1/2), \\
 & | v - I_h v |_{\widetilde H^{1/2}(\Omega)} \leq \frac{C}{\eps} h^{t-1/2-\eps} |v|_{\tH^{t}_{\alpha}(\Omega)} & (s = 1/2),
\end{aligned} \end{equation}
for all $v \in \tH^{t}_{\alpha}(\Omega)$}.
\end{theorem}
\begin{proof}
From the localization estimate \eqref{eq:faermann_new} we obtain
\[
  |v - I_hv|_{H^s(\W)}^2 \leq \sum_{T \in \T} \left[ \iint_{T \times S^1_T} \frac{|(v-I_hv) (x) - (v-I_hv) (y)|^2}{|x-y|^{n+2s}} \diff y \diff x 
  + \frac{2 \w_{n-1}}{sh_T^{2s}} \| v-I_hv \|^2_{L^2(T)} \right].
\]
To shorten notation, for $T \in \T$, we set
\[
\calI_T = \iint_{T \times S^1_T} \frac{|(v-I_hv) (x) - (v-I_hv) (y)|^2}{|x-y|^{n+2s}} \diff y \diff x,
\quad
 \calL_T = \frac1{h_T^{2s}} \| v - I_h v \|_{L^2(T)}^2.
\]
To control the term $\calI_T$, we recall the notation \eqref{eq:defofinteriorandaway} and consider two cases:
\begin{enumerate}[$\bullet$]
  \item $T \in \calT_h^\circ$: \AJS{In this case we apply Proposition~\ref{prop:interpolation_interior} (interior interpolation estimate) 
  and use the mesh grading condition \eqref{eq:mesh_grad} 
  to obtain that
  \[
    \calI_T \leq \frac{C(n,\sigma, t)}{1-s} h^{2(t - s)} \dist(T,\pp\W)^{2(t-s)\frac{\mu-1}\mu} |v|_{H^{t}(S^2_T)}^2.
  \]
  In addition since, for all $x,y \in S^2_T$, we have that $\dist(T,\pp\W) \approx \delta(x,y)$, the right hand side of the previous expression can be modified so that the final estimate reads
  \[
    \calI_T \leq \frac{C(n,\sigma, t)}{1-s} h^{2(t-s)} |v|_{H^{t}_{\alpha}(S^2_T)}^2,
  \]
  where we used the prescribed value for $\alpha$.}
  
  \item $T \in \calT_h^\pp$: \AJS{We now use Proposition~\ref{prop:interpolation_boundary} (boundary interpolation estimate) 
  to arrive at
  \[
    \calI_T \leq \frac{C(n,\sigma, t)}{1-s}h_T^{2(t-s - \alpha)} |v |_{H^{t}_{\alpha}(S^2_T)}^2 
    \le \frac{C(n,\sigma, t)}{1-s}h^{2(t - s)} |v |_{H^{t}_{\alpha}(S^2_T)}^2
  \]
  as a consequence of the grading condition \eqref{eq:mesh_grad} and the prescribed value of $\alpha$}.
\end{enumerate}
Gathering the two previous estimates, \AJS{and using that the constants are} \rhn{non-decreasing in $t$, we deduce}
\begin{equation}
\label{eq:sumfirstterm}
  \sum_{T \in \T} \calI_T \leq C \AJS{h^{2(t - s)} |v|_{{\tH^{t}_{\alpha}(\Omega)}}^2}.
\end{equation}

It remains to control the local $L^2$-interpolation errors $\calL_T$. We again consider two cases:
\begin{enumerate}[$\bullet$]
  \item $T \in \calT_h^\circ$: Employing the error estimate of Proposition~\ref{prop:propsIh} (properties of $I_h$) for $p=2$ we have
  \[
    \calL_T \leq C h_T^{2(t-s)} |v|_{H^{t}(S^1_T)}^2.
  \]
  Then, as in the first case for $\calI_T$, we can use the mesh grading condition \eqref{eq:mesh_grad} and the fact that, for all $x,y \in S^1_T$, $\delta(x,y) \approx \dist(T,\pp\W)$ to obtain
  \[
    \calL_T \leq \AJS{C(\sigma,s) h^{2(t-s)} | v |_{H^{t}_{\alpha}(S^1_T)}^2},
  \]
  \AJS{where we also use the prescribed value for $\alpha$}.
  
  \item $T \in \calT_h^\pp$: Owing to Remark~\ref{rem:s_equals_0} (case $s=0$) we have
  \[
    \calL_T \leq \AJS{C(\sigma) h_T^{2(t - \alpha)} |v|_{H^{t}_{\alpha}(S^1_T)}^2}.
  \]
\end{enumerate}
\rhn{Using the mesh grading condition \eqref{eq:mesh_grad}, the prescribed value of $\alpha$, and the fact that $\mu \in [1,2]$ we see that}
\begin{equation}
\label{eq:sumL2terms}
  \sum_{T \in \T} \calL_T \leq \AJS{C h^{2(t-s)} | v|_{\tH^{t}_{\alpha}(\W)}^2}.
\end{equation}

Adding \eqref{eq:sumfirstterm} and \eqref{eq:sumL2terms} allows us to conclude that
\[
  \AJS{|v-I_hv|_{H^s(\W)} \leq C h^{t-s} |v|_{\tH^{t}_{\alpha}(\W)}},
\]
\AJS{where $\alpha = (t-s)(\mu-1)/\mu$ and $\mu \in [1,2]$}.

Finally, to bound the full $\tHs$-seminorm we need to provide a bound for the term
\[
  \calI_{\mathcal{O}} = \int_\W |(v-I_hv)(x)|^2\int_{\cW} \frac1{|x-y|^{n+2s}} \diff y \diff x \le C(s) \int_\Omega \frac{|(v-I_h v)(x)|^2}{\delta(x)^{2s}} \diff x.
\]
To do so, if $s \neq 1/2$ we employ the inequality
\[
  \calI_{\mathcal{O}} \leq C(s) \begin{dcases}
                               \| v \|_{H^s(\Omega)}^2, & s \in \left(0,\frac12\right), \\
                               | v |_{H^s(\Omega)}^2, & s \in \left(\frac12,1\right),
                             \end{dcases}
\]
whose proof is implicit in the proof of \cite[Corollary 2.6]{AcostaBorthagaray} and uses the fractional Hardy-type inequality of \cite[Theorem 1.1 (T1)]{Dyda} in the case $s>1/2$
\[
\int_\Omega \frac{|w(x)|^2}{\delta(x)^{2s}} \diff x \le C(s) \int_\Omega\int_\Omega
  \frac{|w(x)-w(y)|^2}{|x-y|^{n+2s}} \diff x \diff y
  \quad\forall \ w \in \widetilde{H}^s(\Omega),
\]
and is the content of \cite[Theorem 1.4.4.4]{Grisvard} for $s < 1/2$. 
We point out that, as shown in \cite{LossSloane}, in case $\Omega$ is a convex domain, the constant $C(s)$ in the Hardy-type inequality for $s>1/2$ behaves like $C(s) \approx (s-1/2)^{-2}$ for $s \downarrow 1/2$.
On the other hand, if $s=1/2$, an argument similar to the one provided in the proof of Theorem~\ref{teo:weighted}(weighted regularity of $w_g$) yields for any \AJS{$\eps \in (0,1/4)$}
\[
  \calI_{\mathcal{O}} \leq C \int_\W \frac{|(v-I_hv)(x)|^2}{\delta(x)}\diff x \leq C \diam(\W)^{2\eps} \int_\W \frac{|(v-I_hv)(x)|^2}{\delta(x)^{1+2\eps}}\diff x,
\]
where, in the last step, we used that, since $\W$ is bounded, $\delta(x) \leq \diam(\W)$. It remains to apply, once again, the above fractional Hardy-type inequality \cite[Theorem 1.1 (T1)]{Dyda}. Since this inequality involves the $H^{1/2+\eps}$-seminorm, the constant behaves as $\eps^{-2}$.
\end{proof}

\subsection{The numerical scheme and its analysis}
\label{sub:errest}

Having studied the interpolation operator $I_h$, introduced in Definition~\ref{def:CZinterpolant} (positivity preserving interpolation operator), we can finally proceed to present and analyze the numerical scheme we use to approximate the solution of \eqref{eq:obstacle}. In essence, this is a direct discretization inspired by the approximation of classical obstacle-type problems and their analyses; see \cite{MR0448949, MR3393323}.

We begin by introducing a discrete version of the admissible set as follows:
\begin{equation}
\label{eq:defofKh}
  \K_h = \left\{ v_h \in V_h : v_h \geq I_h \chi \right\}.
\end{equation}
\AJS{Note that, in general, $\K_h \not\subset \K$ and so our approximation scheme is nonconforming}. The discrete problem reads: find $u_h \in \K_h$ such that
\begin{equation}
\label{eq:scheme}
  (u_h, u_h - v_h)_s \leq \langle f, u_h - v_h \rangle, \quad \forall v_h \in \K_h.
\end{equation}
The existence and uniqueness of a solution to \eqref{eq:scheme} is standard. The approximation properties of this scheme are presented below.

\begin{theorem}[error estimate]
\label{thm:conv_rates}
Let $u$ be the solution to \eqref{eq:obstacle} and $u_h$ be the solution to \eqref{eq:scheme}, respectively. Assume that \AJS{$\chi \in \calX(\overline\W)$} satisfies \eqref{eq:defofvarrho} and that $f \in \F(\overline\W)$. \AJS{If $n \geq 2$, $\W$ is a convex polytope}, and the mesh $\T$ satisfies the grading hypothesis \eqref{eq:mesh_grad} with $\mu = 2$ then, \AJS{for $\eps \in (0,s/2)$}, we have that
\[ \begin{aligned}
 & |u-u_h|_{\tHs} \leq \frac{C}\eps h^{1-2\eps} & (s \ne 1/2), \\
  & |u-u_h|_{\widetilde H^{1/2}(\Omega)} \leq \frac{C}{\eps^2} h^{\AJS{1-3\eps}} & (s = 1/2),
\end{aligned} \]
where $C>0$ depends on $\chi$, $s$, $n$, $\Omega$, $\varrho$ and $\| f \|_{\F(\overline\W)}$.
In particular, setting $\eps \AJS{\approx} |\log h|^{-1}$ we obtain
\[ \begin{aligned}
 & |u-u_h|_{\tHs} \leq C h |\log h| & (s \ne 1/2), \\
 & |u-u_h|_{\widetilde H^{1/2}(\Omega)} \leq C h |\log h|^2 & (s = 1/2).
\end{aligned} \]
\end{theorem}
\begin{proof}
After all the discussion about regularity of Section~\ref{sec:regularity} and preparatory steps, the proof of this result follows more or less standard arguments; see \cite[Theorem 4.1]{MR0448949}. However, it requires a combination of Sobolev and H\"older regularity results on the solution as it was first exploited in \cite[Theorems 3.1 and 4.4]{MR3393323}.

We begin by writing
\begin{align*}
  |u-u_h|_{\tHs}^2 &= (u-u_h,u-I_hu)_s + (u-u_h,I_hu-u_h)_s \\ &\leq \frac12 |u-u_h|_{\tHs}^2 + \frac12 |u-I_h u|_{\tHs}^2 + (u-u_h,I_hu-u_h)_s
\end{align*}
so that
\[
|u-u_h|_{\tHs}^2 \leq |u-I_h u|_{\tHs}^2 + 2 (u-u_h,I_hu-u_h)_s.
\]
\AJS{For the first term on the right hand side Corollary \ref{col:globregwithf} (global weighted Sobolev regularity for $f\ne0$) shows that we must
apply Theorem~\ref{thm:interpolation} (global interpolation estimate) with $t = 1+s-2\eps$ and $\alpha =\frac12 - \eps$ to deduce, first of all, that this forces us to set $\mu =2$ and that, in addition, we have}
\[ \begin{aligned}
& |u-I_h u|_{\tHs} \le C h^{1-2\eps} |u|_{\widetilde{H}^{1+s-2\eps}_{1/2-\eps} (\Omega)} \le C \frac{h^{1-2\eps}}{\eps} & (s \ne 1/2), \\
& |u-I_h u|_{\widetilde H^{1/2}(\Omega)} \le \frac{C}{\eps} h^{\AJS{1-3\eps}} |u|_{\widetilde{H}^{\AJS{3/2}-2\eps}_{1/2-\eps} (\Omega)} \le C \frac{h^{\AJS{1-3\eps}}}{\eps^2} & (s = 1/2).
\end{aligned} \]
It remains to bound the second term. To do this we use \eqref{eq:innerprodisduality} to obtain
\[
  (u,I_hu -u_h)_s = \langle \Laps u, I_h u-u_h \rangle.
\]
In addition, since $I_h$ is positivity preserving, we have that $I_h u \in \K_h$ and so it is a legitimate test function for \eqref{eq:scheme}. Adding \eqref{eq:scheme} to the previous equality then yields
\begin{align*}
  (u-u_h,I_hu -u_h)_s &\leq \langle \lambda, I_hu - u_h \rangle = \int_\W \lambda (I_hu - u_h) \diff x\\
      &= \int_\W \lambda (u-\chi) \diff x + \int_\W \lambda (I_h\chi - u_h) \diff x + \int_\W \lambda [I_h(u - \chi) - (u-\chi)] \diff x, 
\end{align*}
where we have used the regularity Theorem~\ref{thm:reg_lambda} (H\"older regularity of $\lambda$) to transform the pairing into an integral. Next, we apply the complementarity conditions \eqref{eq:complementarity} to conclude that $\lambda (u-\chi)=0$. Finally, we use, once again, the complementarity conditions to see that $\lambda \geq 0$ and, since $u_h \in \K_h$, then the middle term is non-positive and can be dropped. Consequently,
\[
(u-u_h,I_hu -u_h)_s \leq \int_\W \lambda [I_h(u - \chi) - (u-\chi)]\diff x = \sum_{T \in \T} \int_{T} \lambda [I_h(u - \chi) - (u-\chi)] \diff x = \sum_{T \in \T} \J_T .
\]

We continue by partitioning the terms in the previous sum into three cases:
\begin{enumerate}[$\bullet$]
  \item $T \subset N$: The complementarity condition \eqref{eq:complementarity} then implies that $\lambda = 0$, whence $\J_T = 0$.
  
  \item $T$ is such that $S^1_T \subset \Lambda$: In this case $u = \chi$ and, again, $\J_T = 0$.
  
  \item $T$ is such that $S^1_T \cap N \neq \emptyset$ and $T \cap \Lambda \neq \emptyset$: The first condition yields the existence of $x_N \in S^1_T \cap N$ for which $\lambda(x_N) = 0$. Since $\lambda\in C^{\AJS{0,1-s}}(\overline{\Omega})$, according to Theorem~\ref{thm:reg_lambda} (H\"older regularity of $\lambda$), we infer that
  \[
    |\lambda(x)|\leq C(\sigma)h_T^{1-s} \quad \forall x \in T.
  \]
  The second condition gives rise to the existence of a point $x_\Lambda \in T$ where $u(x_\Lambda) = \chi(x_\Lambda)$. Using the facts that $u-\chi \in C^{1,s}(\W)$, which can be deduced from Remark~\ref{rem:kappa} (interior regularity estimate), and $T$ is uniformly away from $\partial\W$ because $\varrho>0$ in \eqref{eq:defofvarrho}, we obtain
    \[
    |(u-\chi)(x)| \leq C(\sigma) h_T^{1+s} \quad \forall x \in T.
  \]
  The local stability estimate of Proposition~\ref{prop:propsIh} (properties of $I_h$) with $p=\infty$ then implies
  \[
    |I_h(u-\chi)(x)-(u-\chi)(x)| \leq C(\sigma) h_T^{1+s}.
  \]
  In conclusion, in this case we have
  \[
    \J_T \leq C(\sigma) h_T^2|T|.
  \]
\end{enumerate}
The previous considerations then lead to
\[
  (u-u_h,I_hu -u_h)_s \leq C(\sigma) \sum_{T \in \T} h_T^2 |T|.
\]
Since the mesh grading condition \eqref{eq:mesh_grad} yields $h_T \leq C h$ for all $T \in \T$, this completes the proof.
\end{proof}

\begin{remark}[complexity] \label{rem:conv_rates}
\AJS{Let us take another look at the estimates shown in Theorem~\ref{thm:conv_rates} (error estimate). We will consider two separate cases.

In two dimensions ($n=2)$, since the mesh is assumed to verify the grading condition \eqref{eq:mesh_grad} with $\mu =2 = n/(n-1)$, we have that $\dim V_h \approx h^{-2}|\log h|$, according to Remark~\ref{rem:cardinality} (dimension of $V_h$). This allows us to interpret the assertion of Theorem~\ref{thm:conv_rates} (error estimate) in terms of degrees of freedom as follows
\[ \begin{aligned}
 & |u - u_h |_{\tHs} \leq C (\dim V_h)^{-1/2} (\log \dim V_h )^{3/2} & (s \ne 1/2), \\
 & |u - u_h |_{\widetilde H^{1/2}(\Omega)} \leq C (\dim V_h)^{-1/2} (\log \dim V_h )^{5/2} & (s = 1/2),
\end{aligned} \]
which shows that in this case our method is near optimal.

On the other hand, in three dimensions ($n=3$) we have that
\[
  \dim V_h \approx h^{-4},
\]
see Remark~\ref{rem:cardinality} (dimension of $V_h$). Therefore, the estimate will read
\[ \begin{aligned}
 & |u - u_h |_{\tHs} \leq C (\dim V_h)^{-1/4} \log \dim V_h  & (s \ne 1/2), \\
 & |u - u_h |_{\widetilde H^{1/2}(\Omega)} \leq C (\dim V_h)^{-1/4} (\log \dim V_h )^{2} & (s = 1/2),
\end{aligned} \]
which is not near optimal anymore. One could, in principle, repeat the proof of Theorem~\ref{thm:conv_rates} (error estimate) with $\mu = 3/2$ so that $\dim V_h$ and $h$ have the correct scaling. In this case, however, we need to revisit the weighted regularity estimate for the linear problem \rhn{and the interpolation error estimate \eqref{eq:est_P}. For the former,} instead of Theorem~\ref{teo:weighted} (weighted regularity of $w_g$), we use that
$w_g \in \widetilde{H}^t_\alpha(\Omega)$ for $t < 1+s$ and $\alpha > t - s - 1/2$ (cf. \cite[Proposition 2.2]{BoCi19}).
\rhn{For the latter, we resort to \cite[Lemma 23.1]{Tartar07} to interpolate the error estimates
\[
\| v - P_1 v \|_{L^2(S^2_T)} \leq C h_T^{t-\alpha_1} | v |_{H^t_{\alpha_1}(S^2_T)},
\quad
  \| \nabla(v - P_1 v) \|_{L^2(S^2_T)} \leq C h_T^{t-\alpha_2-1} | v |_{H^t_{\alpha_2}}(S^2_T),
\]
with $\alpha_1 \in [0,1]$ and $\alpha_2\in [0,t-1)$ to obtain \eqref{eq:est_P} with $\alpha\in[0,1-2s+t s)$
\[
 |v-P_1v|_{H^s(S^2_T)} \leq C h_T^{t-\alpha-s} |v|_{H^t_\alpha(S^2_T)}.
\]
This gives a range for $\alpha \in (t-s-1/2,1-2s+t s)$ which turns out to be non-empty for all $s\in(0,1)$. To enforce the condition for $\alpha$ of Theorem \ref{thm:interpolation} (global interpolation estimate) with $\mu=3/2$, i.e.
\[
  \alpha = \left(\frac{\mu-1}\mu\right)(t-s) > t - s -\frac12,
\]
which also satisfies $\alpha<1-2s+t s$ for all $t\in(1,2)$,
we are thus forced to restrict $t < \tfrac34 +s$.} In other words, the full regularity of the solution cannot be exploited, and this would lead to suboptimal error estimates in terms of $h$. In conclusion, in either case we obtain a \rhn{suboptimal convergence rate $(\dim V_h)^{-1/4}$ (up to logarithmic factors) for dimension $n=3$.}} 
\eremk
\end{remark}

\section{Numerical illustrations}
\label{sec:numex}

In this section we assess the sharpness of Theorem~\ref{thm:conv_rates} (error estimate) by displaying the results of numerical experiments performed in two-dimensional domains, and we illustrate the qualitative differences between fractional Laplacians of different orders with an example.

The experiments were carried out with the aid of the code documented in \cite{ABB}; we refer to that work for details on the implementation and a discussion on the challenges that arise when computing the stiffness matrices.
The discrete minimization problems were solved by performing semismooth Newton iterations, as described in \cite[Section 5.3]{Bartels}. A brief explanation on how to construct graded meshes satisfying \eqref{eq:mesh_grad} can be found in \cite{AcostaBorthagaray}.

\subsection{Explicit solution}
\label{ss:explicit}

We first describe how to construct a non-trivial solution to \eqref{eq:obstacle} in the unit ball of $\rn$. For this domain, reference \cite{Dyda2016} explicitly expresses eigenfunctions of an operator closely related to the fractional Laplacian in terms of Jacobi polynomials and an $s$-dependent weight. For example, in dimension $n=2$ and using the Jacobi polynomial $P_2^{(s,0)}$ of degree two
\[
P_2^{(s,0)}(z)
= \frac{4 (s+1)(s+2) +4 (s+2)(s+3)(z-1) + (s+3)(s+4)(z-1)^2}{8},
\] 
define
\[
p^{(s)}(x) = P_2^{(s,0)}(2|x|^2-1), \qquad  u(x) =  \left(1- |x|^2\right)_+^s p^{(s)}(x), \qquad \tilde{f}(x) = 2^{2(s-1)} \Gamma(3-s)^2 \, p^{(s)}(x).
\]
Then, it holds that
\[
(-\Delta)^s u(x) = \tilde{f}(x), \quad x \in B_1.
\]
We now consider a smooth obstacle $\chi$ that coincides with $u$ in $\Lambda=\overline{B_{1/5}}$ and modify $\tilde{f}$ in $B_{1/5}$ so that within this contact set the strict inequality $(-\Delta)^s u > f$ holds. More precisely, we extend $\chi$ to $N=B_1 \setminus \overline{B_{1/5}}$ by using the Taylor polynomial of order two of $u$ on $\pp B_{1/5}$ and set
\[
f(x) = \tilde{f}(x) - 100 \, \left(\frac15- |x| \right)_+ .
\]

\AJS{Note that, as written, $f \notin \F(\overline\W)$. However, at the mesh level, it makes no difference if we smooth out the vertex of the cone $\left(\frac15- |x| \right)_+$ so that we have $f \in \F(\overline\W)$}.
\begin{figure}[ht]
	\centering
	\includegraphics[width=0.45\textwidth]{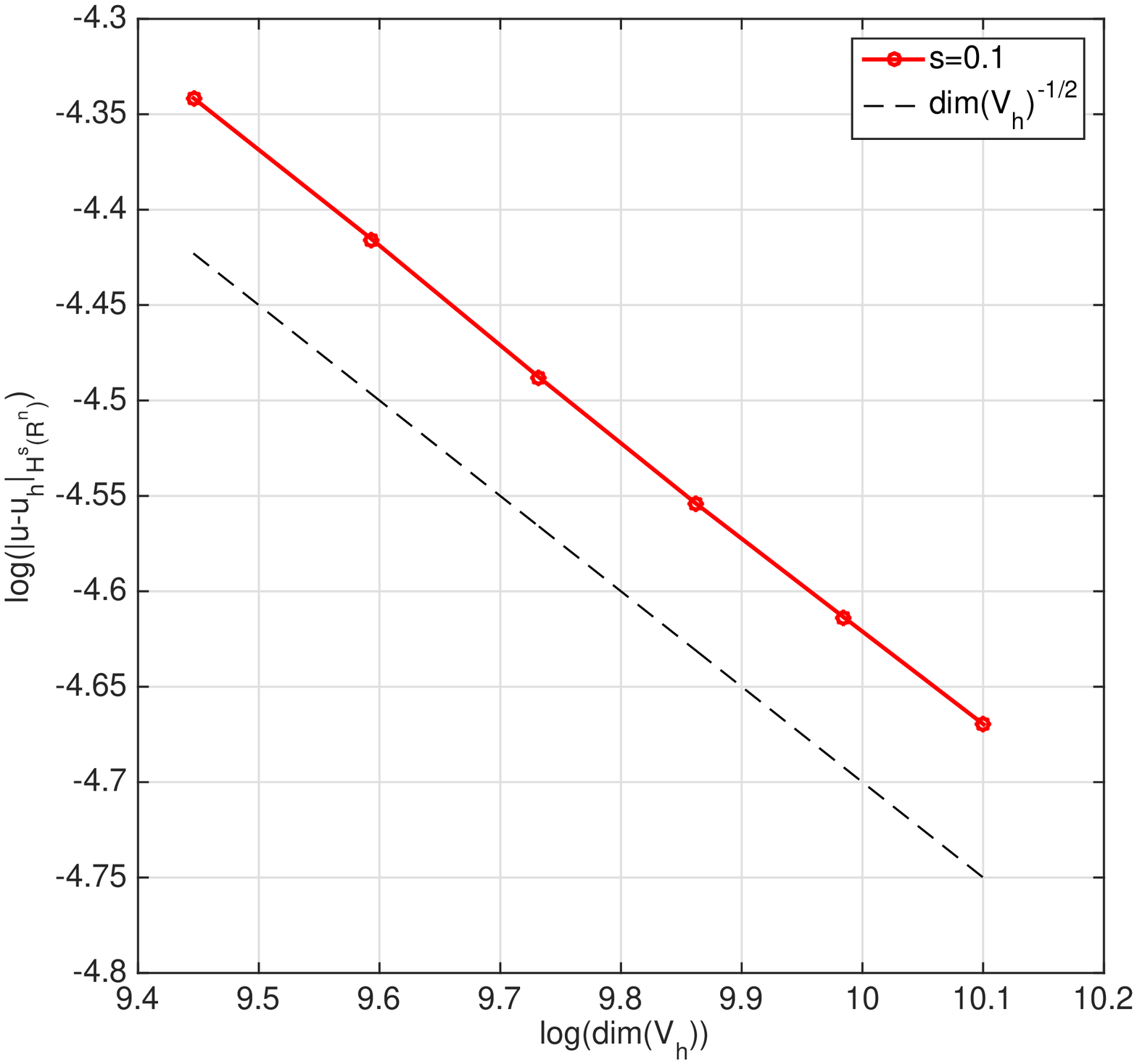} 
	\includegraphics[width=0.45\textwidth]{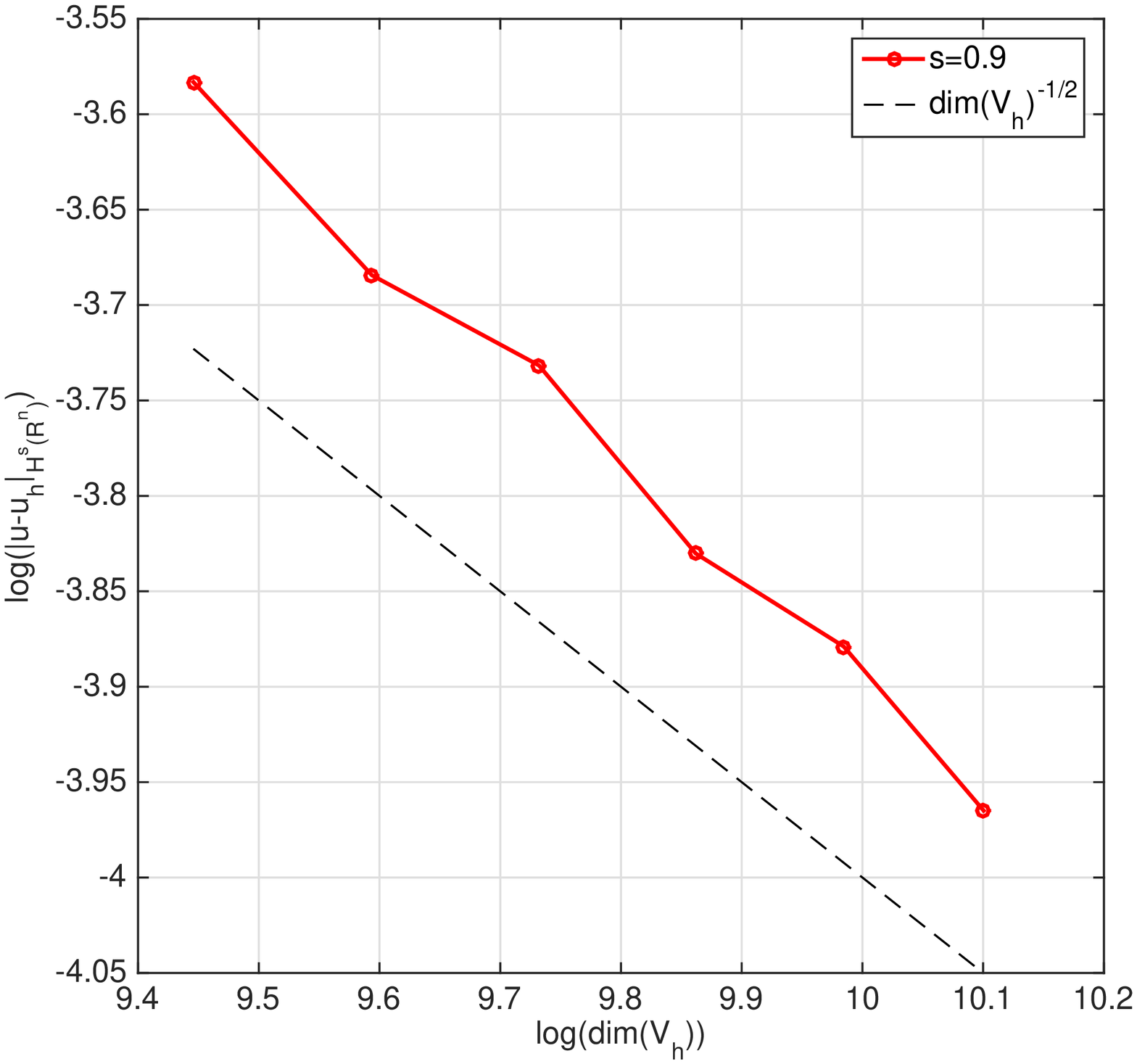} \\
\caption{Computational rate of convergence for the discrete solutions to the fractional obstacle problems described in section~\ref{ss:explicit} over meshes satisfying the grading condition \eqref{eq:mesh_grad} with $\mu=2$. The left panel shows the errors for $s=0.1$ and the right one for $s=0.9$. The rate observed in both cases is $\approx \dim(V_h)^{-1/2}$, in agreement with the theory.
}	\label{fig:conv_rates}
\end{figure}

We carried out computations for $s\in\{0.1, 0.9\}$ using meshes satisfying the grading condition \eqref{eq:mesh_grad} with $\mu = 2$ and different mesh size parameters $h$. 
Figure \ref{fig:conv_rates} shows that the observed convergence rates are in good agreement with either Theorem \ref{thm:conv_rates} (error estimate) or Remark \ref{rem:conv_rates} (complexity).

\subsection{Qualitative behavior} \label{ss:qualitative}

Finally, we consider problem \eqref{eq:obstacle}, posed in the unit ball $B_1 \subset \mathbb{R}^2$, with $f = 0$ and the obstacle 
\[ 
\chi (x) = \frac12 - |x - x_0|, \ \mbox{ with } x_0 = (1/4, 1/4) .
\]
Figure \ref{fig:qualitative} shows computed solutions for $s \in  \{0.1, 0.5, 0.9 \}$ over meshes graded according to \eqref{eq:mesh_grad} with $\mu =2$ and $24353$ degrees of freedom (this corresponds to $h \approx 0.025$). Figure \ref{fig:qualitative} also displays the discrete coincidence set, which contains a neighborhood of the singular point $x_0$. \AJS{After a suitable smoothing of the cone $|x-x_0|$,} both the obstacle $\chi$ and solution $u$ are globally Lipschitz and of class $H^{1+s}(\Omega)$ for all $s\in(0,1)$. We point out that away from $x_0$ but still within the coincidence set $\Lambda$, the obstacle $\chi$ is smooth, say of class $C^{2,1}$, and the regularity and approximation theories developed above apply. In particular, we observe that Theorem \ref{thm:conv_rates} (error estimate) is valid because the only critical point in its proof is the case $S_T^1\subset\Lambda$, for which $u=\chi$ regardless of smoothness.

\begin{figure}[ht]
	\centering
	\includegraphics[width=0.33\textwidth]{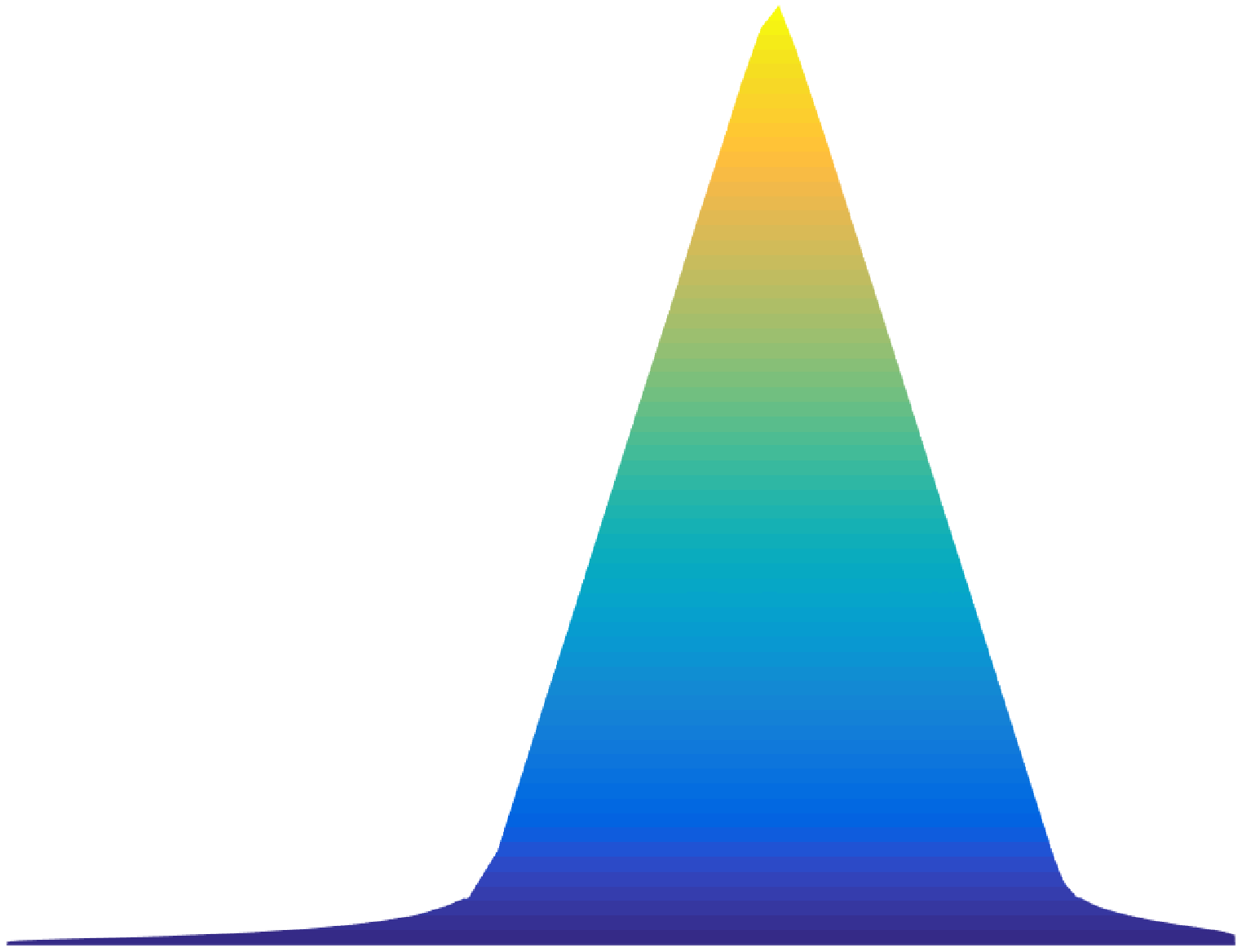}\hspace{-0.5cm} 
	\includegraphics[width=0.33\textwidth]{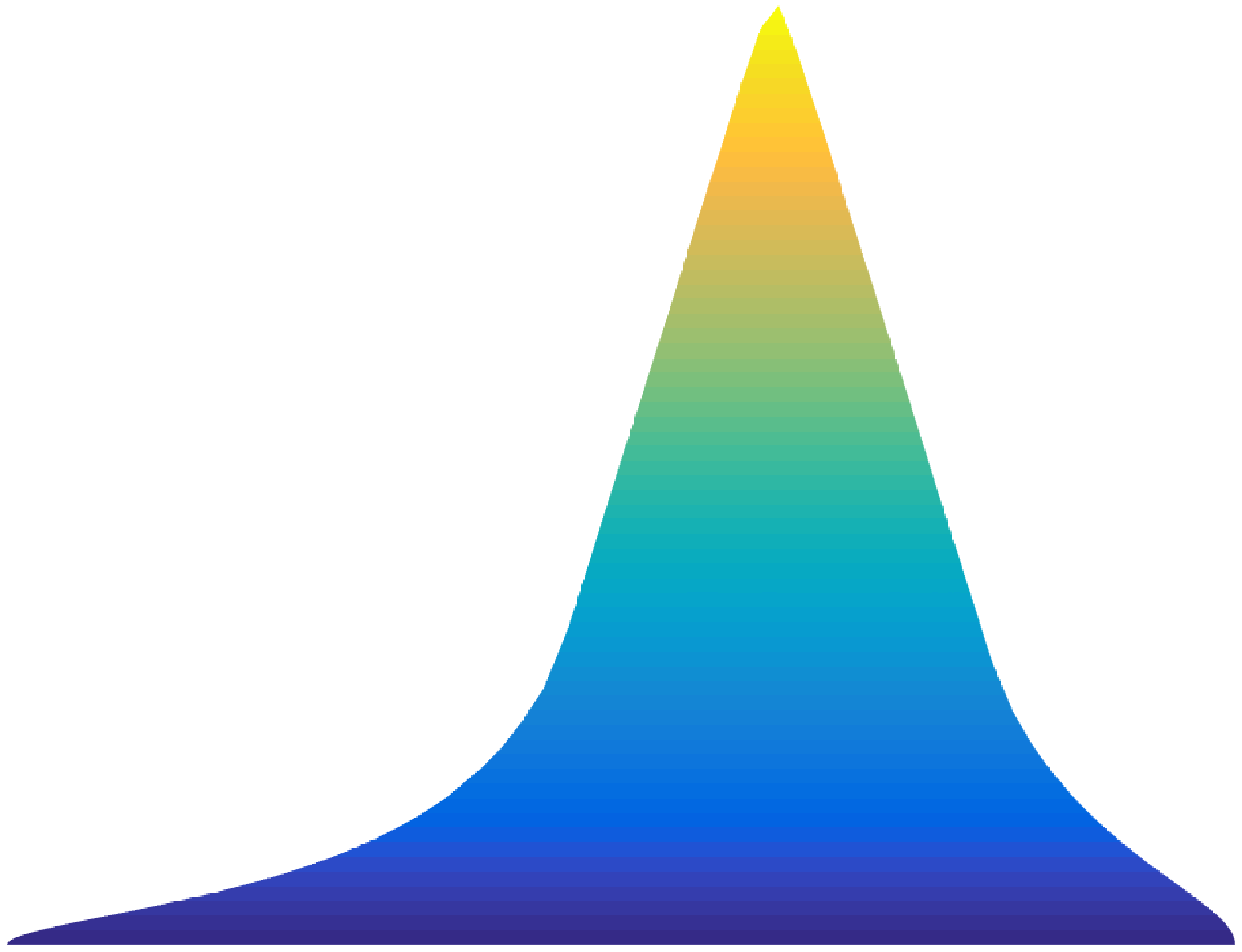} \hspace{-0.5cm}
	 	\includegraphics[width=0.33\textwidth]{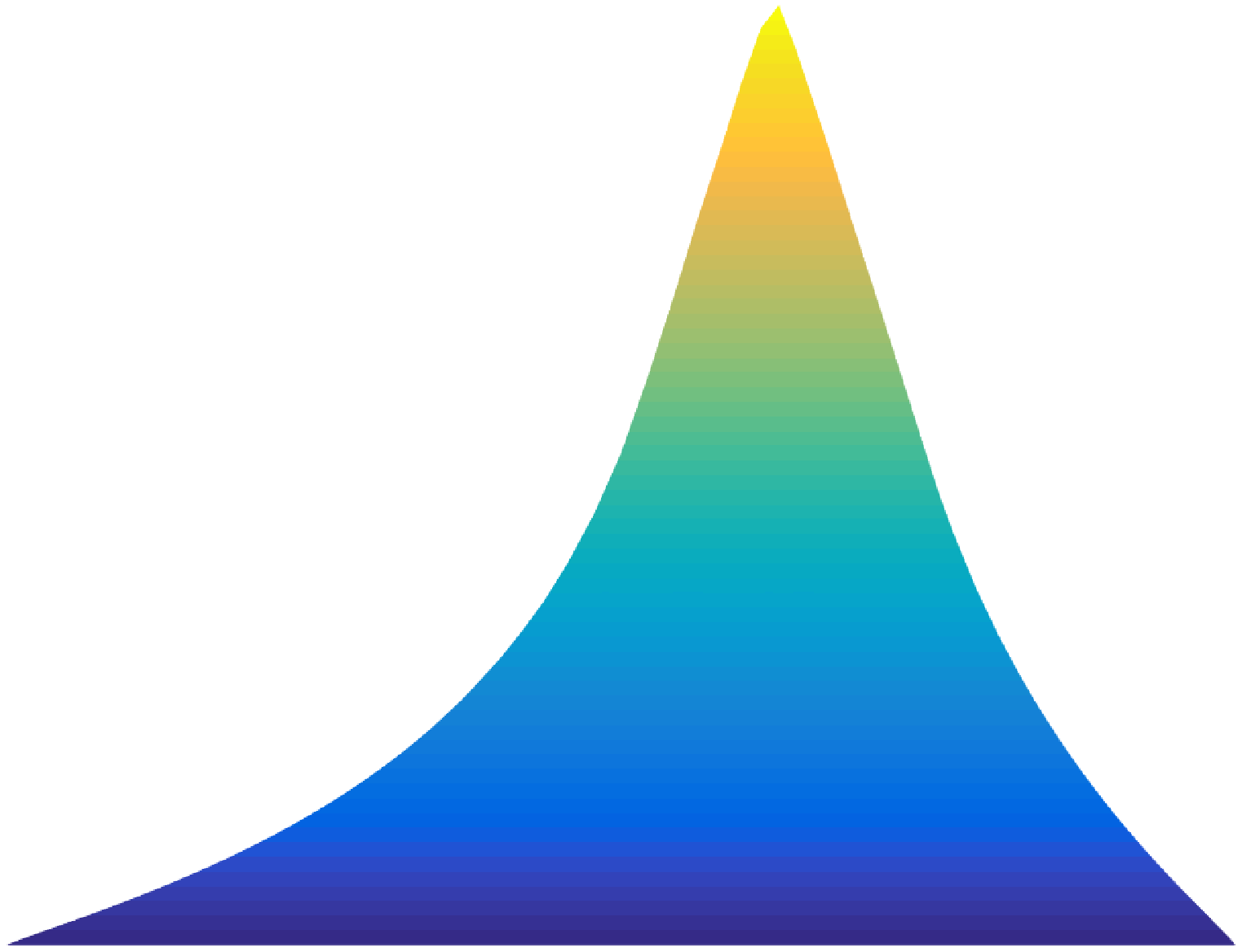} \\
	\includegraphics[width=0.36\textwidth]{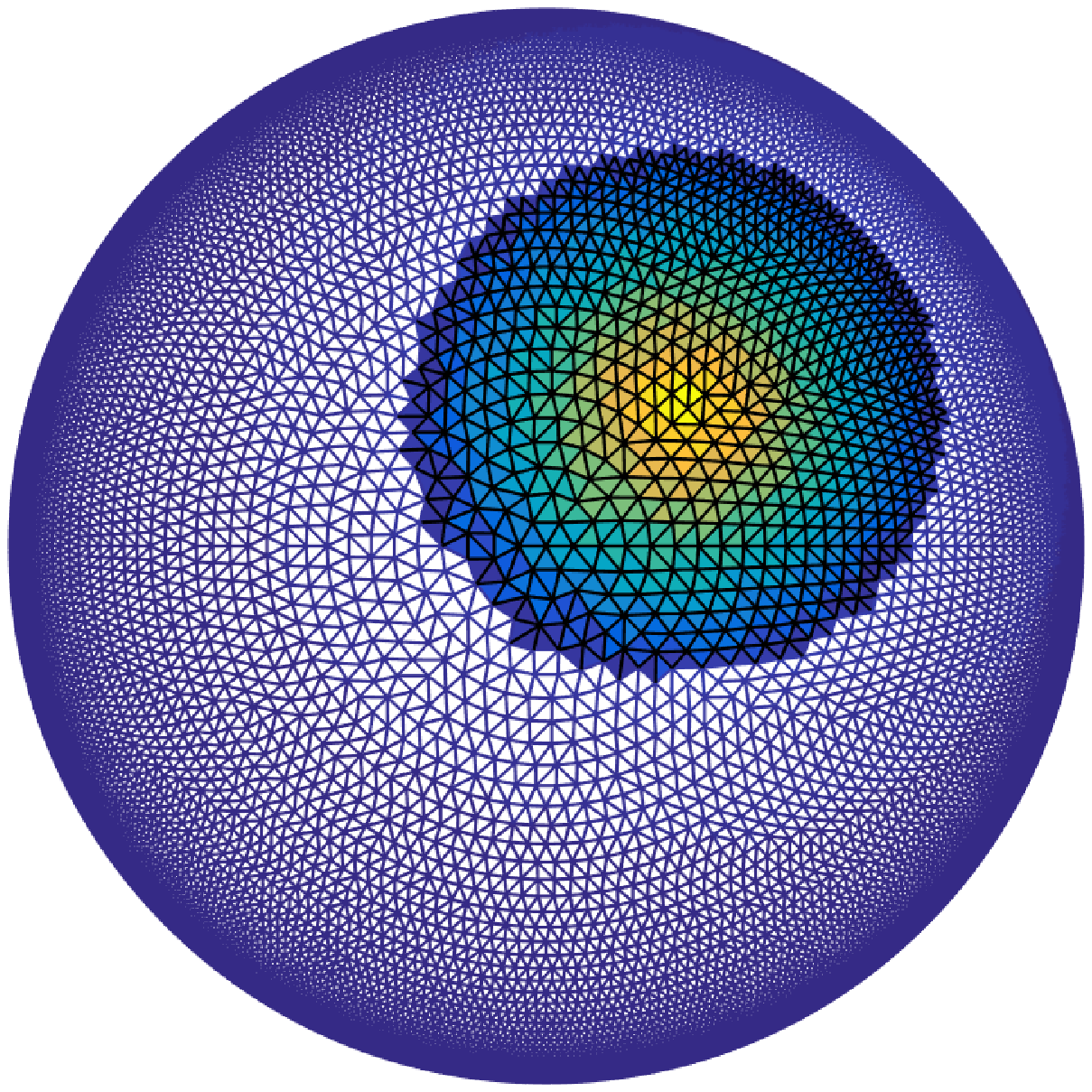} \hspace{-1cm}
		\includegraphics[width=0.36\textwidth]{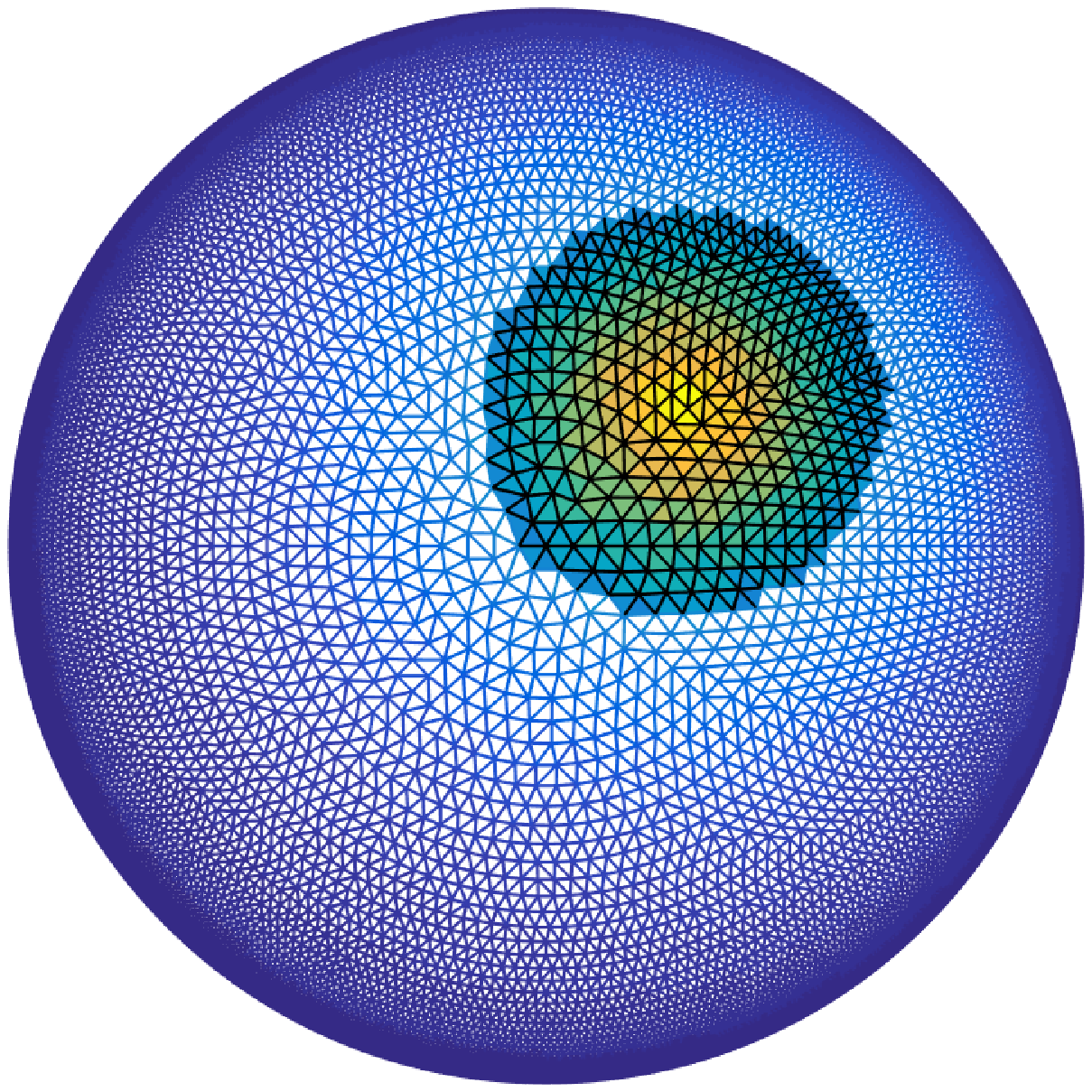} \hspace{-1cm}
	\includegraphics[width=0.36\textwidth]{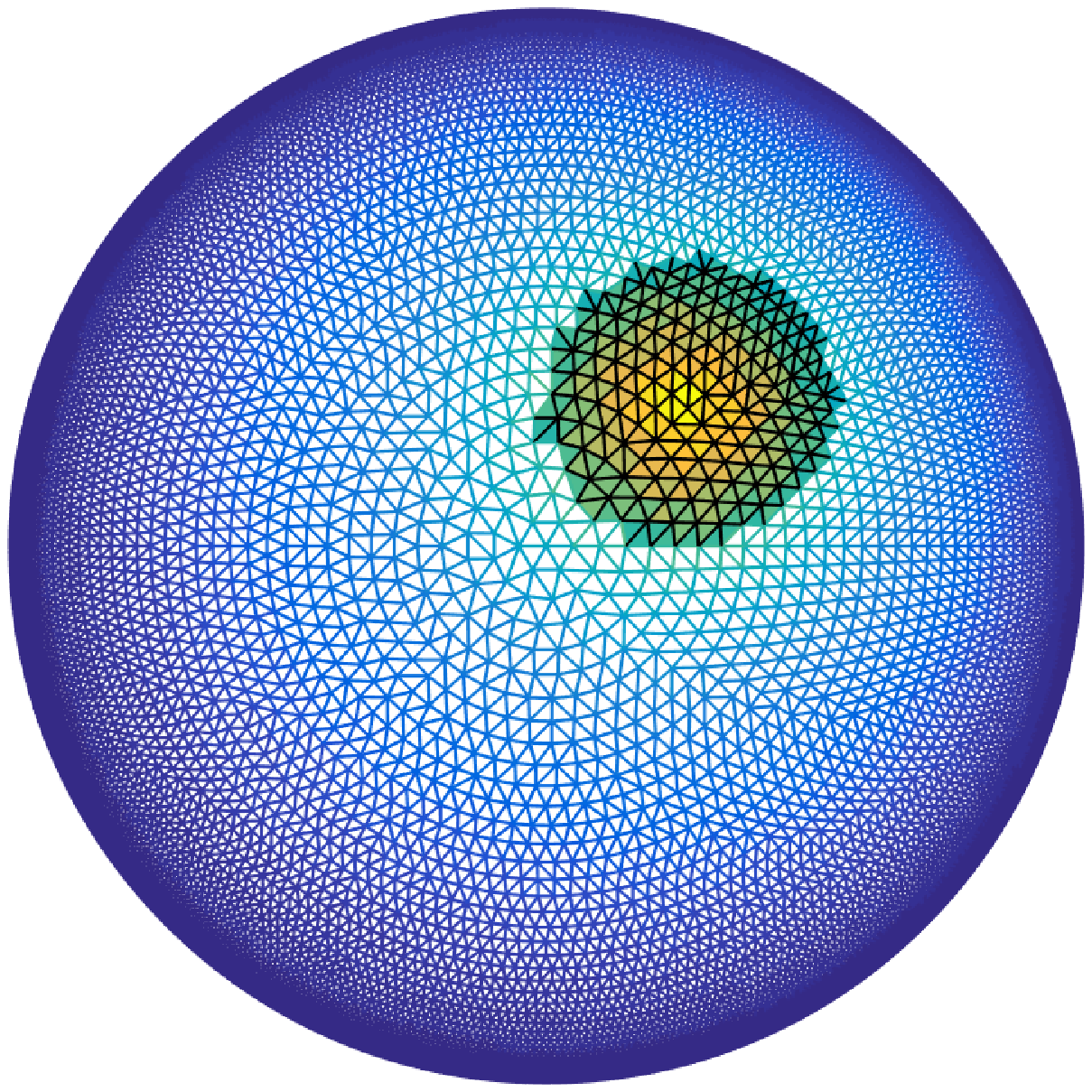}
\caption{Discrete solutions to the fractional obstacle problem for $s=0.1$ (left), $s=0.5$ (center) and $s=0.9$ (right), computed over meshes with $24353$ degrees of freedom, and graded according to \eqref{eq:mesh_grad} with $\mu=2$. Top: lateral view. Bottom: top view, with the discrete contact set highlighted.
}	\label{fig:qualitative}
\end{figure}

Qualitative differences between solutions for different choices of $s$ are apparent. While for $s=0.9$ the discrete solution resembles what is expected for the classical obstacle problem, the solution for $s=0.1$ is much flatter in the non-coincidence set $N$. Moreover, taking into account that the solution of the fractional obstacle problem is non-negative in $\W$ and that $u=\chi_+$ in the formal limit $s=0$, it is apparent that the coincidence set $\Lambda$ decreases with $s$ but always contains $x_0$ in its interior. This fact is verified by the experiments presented in Figure~\ref{fig:qualitative}. We observe that in the diffusion limit $s=1$, the solution is expected to detach immediately for the obstacle away from $x_0$ for a vanishing forcing $f$, whence $\Lambda = \{x_0\}$.

\jp{Finally, Figure \ref{fig:qualitative2} exhibits the convergence rates for these numerical experiments, which are in good agreement with the theoretical predictions. Because we lack \rhn{an explicit} expression for the solution of the obstacle problem in this case, we have used the discrete solutions displayed in Figure \ref{fig:qualitative} as surrogates.

\begin{figure}[ht]
	\centering
	\includegraphics[width=0.95\textwidth]{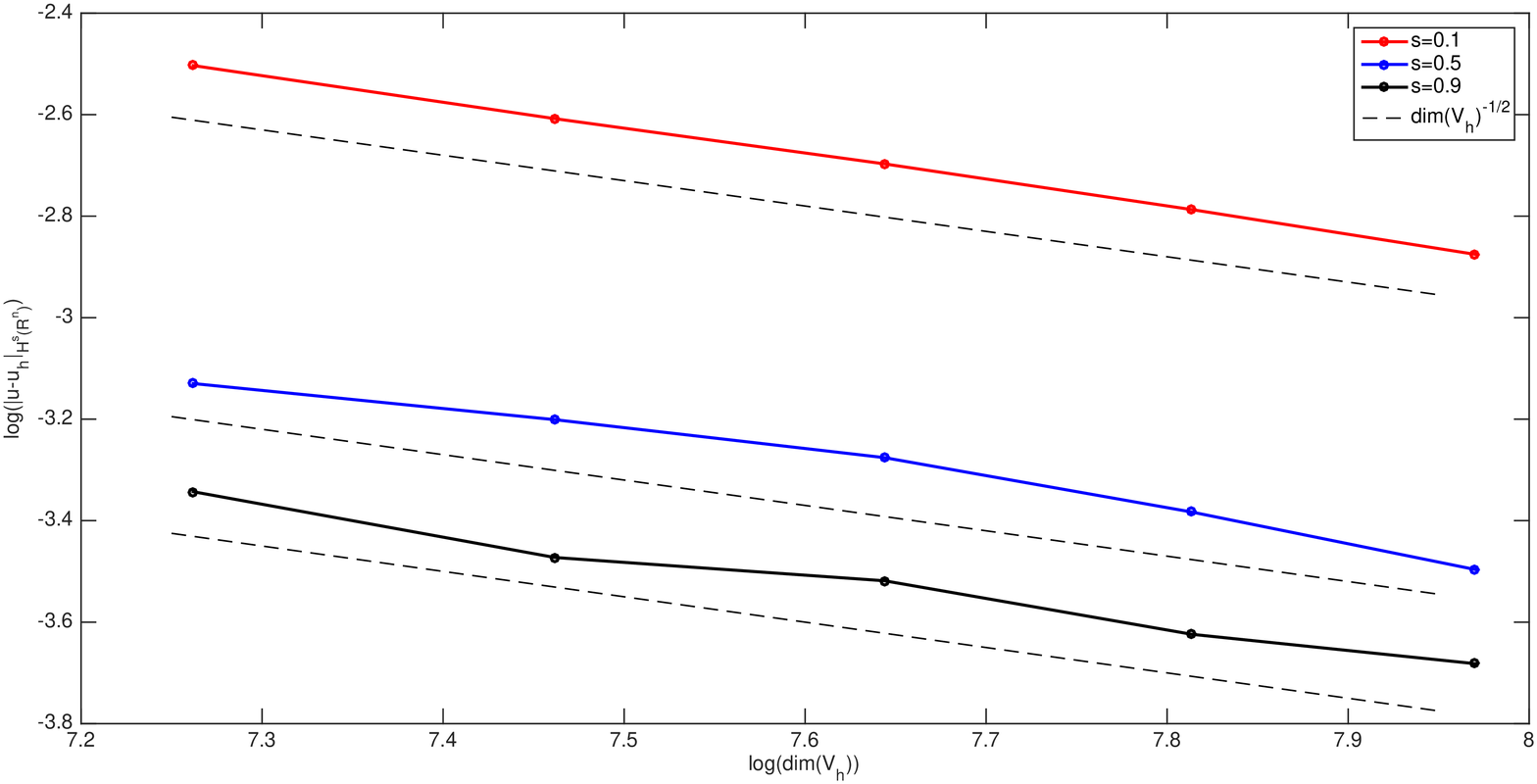} 
\caption{Convergence rates for the experiment described in section \ref{ss:qualitative} with $s=0.1$ (red), $s = 0.5$ (blue) and $s=0.9$ (black). A linear fitting of these data yields estimated convergence rates $0.52$, $0.51$ and $0.47$, respectively.
}	\label{fig:qualitative2}
\end{figure}
}

\bibliographystyle{plain}
\bibliography{bib}

\end{document}